\documentclass[11pt]{amsart}

\usepackage{booktabs}                                 
\usepackage[english]{babel}
\usepackage{latexsym}
 \usepackage[utf8]{inputenc}
\usepackage{babel}
\usepackage{fancyhdr}
\usepackage{longtable}
\usepackage{mathrsfs}
\usepackage{geometry}
\usepackage{textcomp}
\usepackage{caption}
\usepackage{lmodern}
\usepackage{mathrsfs}
 \usepackage[T1]{fontenc}
\usepackage[all,cmtip]{xy}
\usepackage{epsfig}
 \usepackage{amsmath,amsfonts,amssymb,amsthm}
\usepackage{graphics}
\usepackage{graphicx}
\usepackage{verbatim}
\usepackage{dsfont}
\usepackage{enumerate}  

\usepackage{pdflscape}
\usepackage{longtable}
\usepackage{booktabs}
\usepackage{colortbl}
\usepackage[shortlabels]{enumitem}

\newcommand{\C}{\mathbb{C}}

\newcommand{\QQ}{\mathbb{Q}}

\newcommand{\Z}{\mathbb{Z}}

\newcommand{\cO}{\mathcal{O}}

\newcommand{\mQ}{\mathcal{Q}}

\newcommand{\mR}{\mathcal{R}}
\newcommand{\ko}{\mathcal{O}}
\newcommand{\ke}{\mathcal{E}}

\newcommand{\sod}[1]{\langle #1 \rangle}

\newcommand{\PP}{\mathbb{P}}

\newcommand{\mL}{\mathcal{L}}

\newcommand{\of}{\mathcal{O}}

\newcommand{\W}{\bigwedge}

\DeclareMathOperator{\Db}{{D^{\mathrm{b}}}}

\DeclareMathOperator{\Proj}{Proj}

\newcommand{\contr}{\lrcorner}

\DeclareMathOperator{\End}{End}

\DeclareMathOperator{\sort}{sort}

\DeclareMathOperator{\Sym}{Sym}
\DeclareMathOperator{\ddim}{dim}
\DeclareMathOperator{\Gr}{Gr}
\DeclareMathOperator{\OGr}{OGr}
\DeclareMathOperator{\SGr}{SGr}
\DeclareMathOperator{\Hom}{Hom}

\DeclareMathOperator{\van}{van}

\DeclareMathOperator{\TT}{T_1}

\DeclareMathOperator{\Ml}{S_2 Gr}

\newtheorem{thm}{Theorem}[section]
\newtheorem{corollary}[thm]{Corollary}

\newtheorem{lemma}[thm]{Lemma}
\newtheorem{proposition}[thm]{Proposition}
\newtheorem{definition}[thm]{Definition}

\newtheorem{criterion}[thm]{Criterion}
\newtheorem{conj}[thm]{Conjecture}

\newtheorem*{aim*}{Aim of this paper}
\newtheorem*{problem}{Problem}

\makeatletter    
\def\l@subsection{\@tocline{1}{0,2pt}{2pc}{8mm}{\ \ }} 
\makeatletter    
\def\l@section{\@tocline{1}{0,2pt}{2pc}{8mm}{\ \ }} 

\author{Enrico Fatighenti }
\address{Department of Mathematical Sciences\\
Loughborough University\\
  LE113TU, UK}
\email[E.~Fatighenti]{e.fatighenti@lboro.ac.uk}

\author{Giovanni Mongardi}
\address{Dipartimento di Matematica \\
Alma Mater Studiorum - Università di Bologna\\
Piazza di Porta San Donato 5, 40126 Bologna, Italia}
\email[G.~Mongardi]{giovanni.mongardi2@unibo.it}
 
\title[FK3 and IHS]{Fano varieties of K3 type and IHS manifolds}

\begin{document}
\begin{abstract}
We construct several new families of Fano varieties of K3 type. We give a geometrical explanation of the K3 structure and we link some of them to projective families of irreducible holomorphic symplectic manifolds.
\end{abstract}
\maketitle

\tableofcontents

 \section{Introduction}
 Fano varieties and Irreducible Holomorphic Symplectic manifolds (for short, IHS) are two of the most studied classes of varieties in algebraic geometry. They are very different in nature (for example, they have different Kodaira dimensions) and they are often studied using different tools. Indeed, Fano varieties are at the core of birational geometry, while IHS manifolds (sometimes called hyperk\"ahler when the context is more differential-geometric) can be considered as a higher dimensional analogue of K3 surfaces, with lattice theory as one of the most relevant operative tools. 
 One of the most important properties of Fano varieties is their \emph{boundedness}: indeed it is well known that in every dimension there exists a finite number of families of Fano varieties up to deformations. This still holds if we allow some mild singularities, see \cite{birkar} for an up--to--date survey. It is therefore natural to aim for a classification, but such a problem is currently out of reach. A complete answer is known when the dimension is up to three, see for example \cite{ip99} for the smooth case. In the singular case the classification is still an open problem, even in low dimension and with mild singularities (for example terminal).
  From dimension four onwards, only partial results are known. In particular a known explicit bound in terms on the canonical volume is assumed to be hugely non-sharp (being a large number as $(n+2)^{(n+2)^{n2^{3n}}}$), already for $n=2$. The strategy for a partial classification usually is to consider special subclasses of Fano varieties, or fixing some other invariant, such as the \emph{index}. Recall that this the integer $\iota_X$ which is the maximal number for which the anticanonical class is divisible in the Picard group.
   It is a classical result that whenever a Fano $X$ is smooth, the index satisfies $\iota_X \leq \ddim X+1$, with the equality attained only in the case of projective space. Prime Fano varieties, that is Fano varieties with Picard rank $\rho=1$,  of index $\ddim X -2 \leq \iota_X \leq \ddim X+1$ are completely classified, as they are when $\iota_X \geq \frac{\ddim X+1}{2}$ and $\rho_X >1$. Again see \cite{ip99} for a complete list of results. Mukai's conjecture further bounds the Picard rank in terms of the index: namely the conjecture states that for a smooth Fano $\rho_X(\iota_X-1) \leq \ddim X.$ The general philosophy is therefore that high index Fano varieties are somewhat easier to classify than low index.\\
On the contrary the main problem in the study of IHS is the lack of examples. 
 Similarly to the case of Calabi-Yau manifolds, no result of boundedness is known in general for IHS manifolds (although there are some partial results if one fixes for example the Beauville-Bogomolov-Fujiki form). However, finding examples is definitely harder than in the Calabi-Yau case. The known deformation types include two series of examples  for every even dimension found by Beauville for every even dimension (Hilbert scheme of points on a K3 surface and a similar construction, called \emph{generalised Kummer variety} on an abelian surface), and two sporadic examples in dimension 6 and 10, found by O'Grady. Even if we fix the deformation type and we look for \emph{polarised} families (in analogy with the K3 case) the situation does not improve much: very few examples of projective families are known. A survey of this story can be found for example in \cite{beauville}.\\
The interplay between special classes of Fano varieties and IHS manifold is not a new story: a main example is the one of a maximal family of IHS fourfolds (deformation equivalent to the Hilbert Scheme of tow points on a K3 surface) as the Fano variety of lines of a smooth cubic fourfold, due to Beauville and Donagi. We remark that this is not the unique IHS that can be linked to a cubic fourfold, as the recent constructions of Lehn-Lehn-Sorger-van Straten, \cite{llsvs} (an 8-fold of K3$^{[4]}$-type) and Laza-Sacc\`a-Voisin, \cite{lsv} (example of OG10 manifold) highlight. The cubic fourfold is not the only Fano to which we can associate polarised families of IHS: this is indeed a common feature of a special subclass of Fano varieties, called \emph{Fano varieties of K3 type} (FK3 for short) whose study is the central topic of this paper. We give here the key definitions.
\begin{definition} Let $X$ be a smooth, projective $n$-dimensional Fano variety and $j$ be a non-negative integer. The cohomology group $H^j(X, \C) \cong \bigoplus_{p+q=j} H^{p,q}(X)$ (with $j \geq k$) is said to be of $k$ Calabi-Yau type if \begin{itemize} \item $h^{\frac{k+j}{2},\frac{j-k}{2}}=1$; \item $h^{p,q}=0$, for all $p+q=j, \ p <\frac{k+j}{2}$.
\end{itemize}
$X$ is said to be of $k$ (pure) Calabi-Yau type (k--FCY or Fano of k-CY type for short) if there exists at least a positive $j$ such that $H^j(X, \C)$ is of $k$ Calabi-Yau type. Similarly, $X$ is said to be of mixed $(k_1, \ldots, k_s)$ Calabi-Yau type if the cohomology of $X$ has different level CY structures in different weights.
\end{definition}

In the above definition, we consider all sub-Hodge structures, even those naturally arising using Lefschetz's Theorems (for us, a sixfold with a 2 Calabi Yau structure in $H^4$ and $H^4\cong H^6\cong H^8$ is a $(2,2,2)$ Calabi-Yau type). In the paper, we will say that a Fano variety is \emph{central} if all its cohomology groups have level 0.\\ 
A Fano variety of K3 type (FK3) is nothing but a 2-FCY. Fano varieties of CY type were first introduced and studied by Iliev and Manivel in \cite{ilievmanivel}. The authors focus on the case $k=3$, adding moreover an extra condition on the $H^1(T_X)$ (which we do not ask, since it would rule out already the cubic fourfold and many other interesting examples). They classify 3-FCY that can be obtained by slicing homogeneous spaces with linear and quadratic equations. We remark that our definition is purely Hodge-theoretical, but there are deep links with the concept of CY subcategories, see for example \cite{kuzicy}. In particular, constructing examples of Fano varieties of K3 and CY type might help in finding new playground for testing Kuznetsov's conjecture on rationality.\\
We are especially interested in the FK3 case, due to its deep relation with IHS manifolds. Indeed, a result of Kuznetsov and Markushevich  in \cite{kuzm} shows that if $\mathfrak{M}$ is a moduli space  of stable or simple sheaves on $X$, then any form in $H^{n-q-2}(X, \Omega^{n-q})$ defines a closed 2-form in $H^0(\mathfrak{M}^{\textrm{smooth}}, \Omega^2)$. This is indeed a good starting point in the hunt for examples of IHS. In particular, let us mention the IHS linked to the Debarre-Voisin twentyfold hypersurface, or to a Gushel--Mukai fourfold, or to a section of a product of $\PP^3$, all examples of FK3 varieties, see \cite{debarrevoisin}, \cite{dk16}. \cite{ilievmanivel2}. \\
Although FK3 are definitely easier to hunt than IHS, there are not many known examples in the literature. For example, as complete intersections in (weighted) projective spaces one finds only the cubic fourfold, see \cite{ps18}. More examples are found if one allows terminal and $\QQ$-factorial singularities, see \cite{frz19} but no new examples of IHS are produced anyway. In \cite{eg1} we conjectured that even taking complete intersection in Grassmannian one does not get any new example other than a complete intersection with four linear hypersurfaces in the Grassmannian $\Gr(2,8)$ and the above mentioned examples. This paper deals with the construction of examples of FK3 as zero locus of general global section of homogeneous vector bundles in Grassmannians or products of such. This is motivated by the list of K\"uchle, see \cite{kuchle}, of index 1 Fano fourfolds obtained in such a way, where few more interesting FK3 are found.  Therefore the aim of this paper is twofold:
\begin{aim*} \begin{enumerate}\item Construct new examples of Fano varieties of K3 type;
\item Construct examples of polarised families of IHS from our FK3.
\end{enumerate}
\end{aim*}
This paper is a first step of this project. One of its main aim is to show that there might be a lot of examples \emph{out there}. Even if obtaining a complete classification of all Fano of K3 type might be out of reach, a classification might be attainable if we restrict ourselves to some special subcases, for example Fano obtained as zero locus on Grassmannians and homogeneous varieties. The main problem here is that in general translating the (Hodge-theoretical) requirement of being of K3 type into algebraic conditions is not easy. Using some tools that we developed in \cite{eg1} we were anyway able to find some numerological condition useful to produce examples of FK3, see Numerology \ref{num}. Unfortunately the condition in \ref{num} are still too general for replicating a classification-type argument as the original one from K\"uchle. However, \ref{num} has the advantage of highlighting the connection between FK3 and \emph{central} Fano varieties, that is Fano such that $h^{p,q}\neq 0$ if and only if $p=q$.  When this happens, we say that all the cohomology groups of $X$ has \emph{level} (lv) 0, see \ref{num}. Indeed a future problem we are interested in is the following, possibily up to restriction to some special subcases, with zero locus of sections of homogeneous vector bundles as a first step.
\begin{problem} Classify Fano varieties such that lv $(H^j(X,\C))=0$ for all $j$.
\end{problem}
\subsection{How we subdivide the examples}
 We first write down the list of examples that we have found. Later on in the paper we will explain the numerology behind our list, and give a detailed geometrical description of our examples. Our purpose its twofold. Indeed to a Fano of K3-type we want to associate (whenever possible) both a K3 category and an IHS manifold. For the definition of K3 or CY (sub) category we follow \cite{kuzicy}. Before doing this, we need to prove first that the families of Fano that we consider are of K3 type. This is done usually with either Riemann-Roch type computations as for example in \ref{m7} or using our Griffiths ring-type construction as in Proposition \ref{s1}, or via a Borel-Bott-Weil computation, as in Proposition \ref{t129}.
  In particular we divide our list into three distinct blocks. We say that a FK3 $X$ is of \emph{blow-up} type (\textbf{B}) if there exists a pair $(Y,S)$, with $S \subset Y$, $Y$ Fano, $S$ K3 surface such that $X \cong Bl_S Y$. Examples of this type are already included in K\"uchle list, \cite{kuchle}, called \emph{c7} and \emph{d3}. We say that a FK3 $X$ is of \emph{Mukai type} (\textbf{M}) if we can reduce systematically the study of its derived category to Mukai's classification of Fano threefolds. We say that a FK3 $X$ is \emph{sporadic} (\textbf{S}) if it does not fall in one of the two previous categories. We collect all our list of examples of FK3 in Table \ref{table}. \\
 For FK3 of blow-up and Mukai type the question on the existence of a K3-subcategory admits always a positive answer. This is the content of  Propositions \ref{lem:blowup}, \ref{rennemo} and Theorem \ref{cayley}. However the question of existence of an IHS linked to any FK3 is far from being answered. We give an example in Proposition \ref{gp2} . For the FK3 of sporadic type, we do not have any information a priori. For all of them the question on the existence of a K3-subcategory is open, and even we have to cook up ad-hoc methods even to show that they are of K3 type (in the Hodge theoretical sense). Here as well there is no easy answer from the IHS viewpoint. A new construction is given for example in Proposition \ref{o2}. Special attention must be placed upon example (\textbf{S6}) and (\textbf{S7}). Indeed they are cut by irreducible vector bundles which are not linear. We observe as well the appereance of mixed structures of $(2,3)$-CY type. The last part of the paper is devoted to the study of these varieties. The results about IHS are collected in Table \ref{table2}. We point out that we believe that to any of the example in Table \ref{table} we will eventually be able to construct an example of polarised IHS. We added in both our tables two examples found independently by Iliev and Manivel in \cite {ilievmanivel2}, while our work was still in the very early stage. These are the families \textbf{B1} and \textbf{S3}. Although they were already known we decided to include them anyway in our list, since they fit perfectly in our pattern. 

We highlight now the main results and the structure of this paper.
\subsection{Results and Structure}
This paper is devoted to the construction of a meaningful bunch of examples of Fano varieties of K3 type. We mainly exploit our numerological condition in \ref{num}, coming from a similar analysis to the one we carried out in \cite{eg1}. Our main result can be summarised in
\begin{thm}There exists 23 examples of families of Fano varieties of K3 type obtained as zero locus of general global section of homogeneous vector bundles over Grassmannians or products of such. These Fano varieties have dimension $4 \leq n \leq 20$, Picard rank $1 \leq \rho_X \leq 3$ and index $\frac{n-1}{2}\leq \iota_X \leq \frac{n}{2}$.
\end{thm}
See Table \ref{table} for the list of this Fano varieties. We point out that since they have an index which is comparatively high with respect to the dimension (close to Wisniewski's bound), there could be hope for a classification.
For each of this Fano we first needed to prove that they are of K3 type. We either explain geometrically in a systematic way (whenever possible) the presence of K3 structure (both from a Hodge-theoretical and derived category viewpoint) or we give an ad-hoc description for the sporadic cases. We point out that new examples may and will be discovered and analysed in a series of future works.\\
Some of the Fano we analyse have new and interesting behaviours. We collect some of the results here.
\begin{thm} There exists prime Fano varieties with multiple CY structures (see Proposition \ref{3k3}) and with mixed Calabi-Yau (2,3) structure, (see Proposition \ref{23cy}).
\end{thm}
To the best of our knowledge, these are the first examples of known prime Fano varieties with this property. The prime hypothesis eliminates the possibility for these CY structure to come from a blow-up, a projective bundle or other related constructions. We link some of these Fano varieties to projective families of IHS manifolds. Unfortunately, up to now we have only found new ways of describing old examples, but we believe that a further extensive examination of our list could lead to new constructions. We collect our results here.
\begin{thm}We show that the Hilbert square on a K3 of genus 8 is isomorphic to the zero locus of a certain bundle on $\Gr(4,6) \times \Gr(2,6)$, see Proposition \ref{gp2}. We show that the Debarre-Voisin IHS 4-folds are isomorphic to the space of special rational fourfolds on varieties of type $\TT(2,10)$, see Proposition \ref{t2} and to the compactification of the space of $(\PP^1)^3$ on a linear section of $\Ml(3,8)$, see Theorem \ref{hk}.
\end{thm}
These results are collected in Table \ref{table2}. We spend a few words on the structure of this paper. In \textbf{Section 2} we explain how our numerological condition creates the list and we 
explain some straightforward geometric tricks and a general strategy to attack these Fano varieties. In \textbf{Section 3} we perform a case--by--case analysis of the most interesting examples and we prove our main results. We finish with a bunch of \textbf{Appendices}, where we describe three related cases we encountered: some extra Fano  varieties of 3CY type, a trio of infinite series of Calabi--Yau varieties and a Fano variety with a fake K3 structure.
 \subsubsection*{Acknowledgements}
 This paper was completed throughout the course of the past year and a half. The work was carried out mainly at Roma Tre and Bologna University, in several of its campus sites (although some of the latter were not officially recognised by its own administration). Many people gave useful comments and suggestions throughout the whole process. We mention in particular Atanas Iliev, for sharing with us some of the ideas that led to Theorem \ref{hk}, and Alexander Kuznetsov, for many suggestions, ideas shared and comments on an early draft of this manuscript. Many of the computations were carried out using a Macaulay2 code written by the first author together with Fabio Tanturri. We thank as well for various ideas, conversations and support Hamid Ahmadinezhad, Vladimiro Benedetti, Marcello Bernardara, Daniele Faenzi, Lorenzo Federico, Camilla Felisetti, Michal and Grzegorz Kapustka, Laurent Manivel, Luca Migliorini, Claudio Onorati, Miles Reid and J\o rgen Rennemo. EF was supported by MIUR-project FIRB 2012 "Moduli spaces and their applications" and by an EPSRC Doctoral Prize. GM was supported by ``Progetto di ricerca INdAM per giovani ricercatori: Pursuit of IHS''.  Both authors are member of the INDAM-GNSAGA and received support from it.
 
 \section{The quest for examples}

 \textbf{Notation for the paper and for the tables}
 With $\mR$ and $\mQ$ we denote (respectively) the rank $k$ tautological and the rank $n-k$ quotient bundle on the Grassmanian $\Gr(k,n)$. We fix the convention that $\of_G(1) =\mathrm{det}(\mQ)=\mathrm{det}(\mR^{\vee})$.
 $\mathrm{S}_i\Gr(k,n)$ denotes the $i$-th symplectic Grassmannian. The most relevant cases for us are for $k=1$ and $k=2$. For $i=1$ this variety is nothing but the usual symplectic Grassmannian (usually called Lagrangian when $2k=n$), for $i=2$ it is the \emph{bisymplectic Grassmannian}, which will be better defined and characterised later in the paper. If $k=1$ we will simply write $\SGr(k,n)$. $\overline{\SGr(3,6)}$ in the table will denote a linear section of $\SGr(3,6)$. $\OGr(k,n)$ denotes the orthogonal Grassmannian and $\mathbb{S}_{n} $ we denote one of the two connected components of $\OGr(n,2n)$ in its spinor embedding.
 $\TT(k,n)$ denotes the subvariety of $\Gr(k,n)$ cut by the zero locus of a general three-form $\sigma \in \W^3 V_n^\vee$. According to $k$, $\TT(k,n)$ can be represented as the zero locus of a general global section of a different vector bundle. As an example, if $k=3$, $\TT(3,n)$ is nothing but a linear section of the Grassmannian $\Gr(3,n)$, if $k=2$, $\TT(2,n)$ is the congruence of lines given by the bundle $\mQ^*(1)$ and if $k=4$ the bundle is of course $\W^3 \mR^{\vee}$.\\
 We use $X_1 \subset G$ to denote a linear section of the variety $G$ (and similarly for higher degree or multidegree). Whenever there might be ambiguity or we want to emphasize the choice of the linear subspace we might write $X_H$. Similarly, sometimes we will use the shorthand $X_{\mathcal{F}} \subset G$ to denote the zero locus of a general global section of the vector bundle $\mathcal{F}$ over $G$.\\
 The notation $H^n_{\van}(X)$ (and similarly for the $(p,q)$ part) will denote the vanishing subspace of the cohomology group, see \cite[2.27]{voisin2} for a definition. \\
 If $X$ and $Y$ are smooth projective variety we will use the shortand $D^b(X) \hookrightarrow D^b(Y)$ to mean that one can construct a semiorthogonal decomposition for $D^b(Y)$ where $D^b(X)$ appears as one of the factors, up to a fully faithful functor.\\
  The notation $S_g$ means a K3 surface of genus g. With $Q_k$ we indicate the $k$-dimensional quadric hypersurface.

\begin{table}[ht]
\centering
\begin{tabular}{@{} *9l @{}l @{}l @{}l @{}l @{}l @{}}    \toprule
no. & \emph{$X \subset Y$}& $\ddim X$ & $\iota_X$ & $\rho_X$ &  Comments \\ \midrule
B1 & $X_{(2,1,1)} \subset \PP^3 \times \PP^1 \times \PP^1$ &4& 1 & 3&  $X \cong Bl_{S_7} (\PP^3 \times \PP^1)$\\
B2 & $X_{(2,1)} \subset \Gr(2,4) \times \PP^1$& 4&1&2&  $X \cong Bl_{S_5} \Gr(2,4)$\\
M1 & $X_{(1,1,1)} \subset \PP^3 \times \PP^3 \times \PP^3$& 8& 3&$ 3$ &$D^b(S_{3} )\hookrightarrow D^b(X)$ \cite[Section 4]{ilievmanivel2} \\
M2 & $X_{(1,1,1)} \subset Q_3 \times \PP^2 \times \PP^2$ & 6 & 2& 3& $D^b(S_{4}  )\hookrightarrow D^b(X)$\\
M3 & $X_{(1,1)} \subset \Gr(2,5)\times Q_5$& 10 &4& 2&  $D^b(S_6) \hookrightarrow D^b(X)$ \\
M4 & $X_{(1,1)} \subset \SGr(2,5) \times Q_4$&8&3& 2&$D^b(S_6) \hookrightarrow D^b(X)$  \\
M5&  $X_{(1,1)} \subset \Ml(2,5) \times Q_3$& 6&2&2& $D^b(S_6) \hookrightarrow D^b(X)$\\
M6& $X_{(1,1)} \subset \mathbb{S}_{5} \times \PP^7$& 16& 7& 2&$D^b(S_7) \hookrightarrow D^b(X)$\\
M7 & $X_{(1,1)} \subset \Gr(2,6) \times \PP^5$& 12 & 5& 2 &  $D^b(S_{8}) \hookrightarrow D^b(X)$\\
M8& $X_{(1,1)} \subset \SGr(2,6) \times \PP^4$& 10& 4& 2&$D^b(S_8) \hookrightarrow D^b(X)$\\
M9& $X_{(1,1)} \subset \Ml(2,6) \times \PP^3$& 8 &3& 2 & $D^b(S_8) \hookrightarrow D^b(X)$\\
M10 & $X_{(1,1)} \subset \SGr(3,6) \times \PP^3$&8&3& 2&$D^b(S_9) \hookrightarrow D^b(X)$  \\
M11 & $X_{(1,1)} \subset \overline{\SGr(3,6)} \times \PP^2$&6&2& 2&$D^b(S_9) \hookrightarrow D^b(X)$  \\
M12& $X_{(1,1)} \subset \mathrm{G}_2 \times \PP^2$& 6& 2& 2& $D^b(S_{10}) \hookrightarrow D^b(X)$\\ 
M13& $X_{(1,1)}\subset \Gr(2,8)\times \PP^3$ & 14 & 1 & 2 & $D^b(S_3) \hookrightarrow D^b(X)$\\
S1 & $X_{1^4} \subset \Gr(2,8)$& 8 & 4& 1 & $D^b(S_3) \hookrightarrow D^b(X)$ \\
S2& $X_1 \subset \OGr(3,8)$ &8&3& 2&  $D^b(S_7) \hookrightarrow D^b(X)$\\
S3& $X_1 \subset \SGr(3,9)$ &14&6& 1 & \cite[Section 5]{ilievmanivel2}\\
S4& $X_1 \subset \Ml (3,8)$ &8&3& 1 & \\
S5& $X_1 \subset \TT(2,9)$ &6&2& 1 &\\
S6& $\TT(2,10)$ &8&3& 1 &$ 3 \times $ K3 structure\\
S7& $X_1 \subset \TT(2,10)$ &7&2& 1 & $ 2 \times $ K3 structure, $ 1 \times $ 3CY\\
S8& $X_{L} \subset \TT(k,10)$ &&& 1 &  invariants depending by $k$ and $L$\\
    \bottomrule
 \hline
\end{tabular}
\captionof{table}{Fano of K3 type with invariants} \label{table}
\end{table}

\begin{table}[ht]
\centering
\begin{tabular}{@{} *9l @{}l @{}l @{}l @{}l @{}l @{}}    \toprule
no. & \emph{$X \subset Y$} & IHS $Z$& Comments\\ \midrule
M1 & $X_{(1,1,1)} \subset \PP^3 \times \PP^3 \times \PP^3$ & \cite[Section 4]{ilievmanivel2} & $Z \cong Hilb^2 S_3$ \\
M7 & $X_{(1,1)} \subset \Gr(2,6) \times \PP^5$& Prop.\ref{gp2}& $Z \cong Hilb^2 S_8$\\
S2& $X_1 \subset \OGr(3,8)$ & Prop.\ref{o2}&$Z \cong S_7$\\
S3& $X_1 \subset \SGr(3,9)$ &  \cite[Section 5]{ilievmanivel2}& $Z \cong Z_{DV}$\\
S4& $X_1 \subset \Ml (3,8)$ & Thm. \ref{hk} &$Z \cong Z_{DV}$\\
S6&$\TT(2,10)$ & Prop. \ref{t2}&$Z \cong Z_{DV}$\\
    \bottomrule
 \hline
\end{tabular}
\captionof{table}{Projective families of IHS linked to FK3} \label{table2}
    \end{table}

\subsection{What are we looking for?}
Many of the examples in the above table are obtained by chasing up the same numerology. Indeed from arguments similar to the one used in \cite{eg1} one can come up with a numerical criterion (cf. \cite{ilievmanivel} and \cite{kuzicy} for comparison and similar criteria). For a smooth projective variety we define the \emph{level} of $H^j(X, \C)$ as the largest difference $|p-q|$ for which $H^{p,q}(X) \neq 0$, with $p+q=j$.
It is obvious that lv$(H^j(X, \C)) \leq $ wt $(H^j(X, \C)) \leq \ddim X$. For a Fano variety by Kodaira vanishing the first inequality is always strict. For example, if $X$ is a Fano of dimension $n$, then $\mathrm{lv} (H^n(X,\C)) \leq \ddim X-2$. Moreover we say that a variety $X$ is \emph{central} if all of its $H^j$ have level zero, or equivalently if $h^{p,q}(X)=0$ for $p \neq q$.
\begin{criterion}\label{num} Let $Y$ be a smooth projective Fano variety of dimension $2t+1$ and index $\iota_Y$. 
Assume that $t$ divides $\iota_Y $ and that lv$(H^{2t+1}(Y))\leq 1$. Then a generic $ X \in | -\frac{1}{t} K_Y |$ is a Fano variety of K3 type, with the K3-type structure located in degree $2t$.
\end{criterion}
The above criterion is not necessary. Notable exceptions are \textbf{(S1)} (where the divisibility relation does not hold) and \textbf{(S6)}, where there the decomposition in irreducibles of the bundle that cuts the variety has no linear factor (albeit the variety has the correct ratio between dimension and index), and moreover two K3 sub-Hodge structures are present, in degree 6 and 8.\\
To the best of our knowledge the above numerology admits no counterexample. However the cohomological vanishing required to potentially prove the statement are ad-hoc, and there seems to us no easy way to transform the above statement in a proper theorem. However it is a cheap and easy way to produce several candidates, which turn out to be all of the desired type. We do not feel confident enough to state it as conjecture, as it stands. There could be ways of turning it into a statement or a conjecture. For example we could ask for $Y$ to have a rectangular Lefschetz decomposition in the categorical sense. Or, whenever $Y$ itself is cut by a section of an homogeneous vector bundle $\mathcal{F}= \bigoplus \mathcal{F}_i$ on $\Gr(k,n)$, we might ask that the slope $\mu (-\frac{1}{t} K_Y) > \mu(\mathcal{F}_i)$ for all $i$. However, for the purpose of the current paper, we prefer to leave it as is is, and we plan to formalise this statement in a future work.
\subsubsection{Some numerology (and how the list is created)}
The list of FK3 in the tables has no presumption of being complete. The main problem is the condition on the level of Hodge theory of the ambient variety, which is quite hard to control. The first case to investigate is the one of complete intersections in homogeneous varieties. We conjectured in \cite{eg1} that  there are no more FK3 as complete intersection in $\Gr(k,n)$ other than the well-known cubic fourfold, the Gushel-Mukai fourfold, the Debarre-Voisin twentyfold hypersurface and a codimension four linear section of Grassmannian $\Gr(2,8)$. We have not been able to prove this conjecture yet, but no counterexample has been found either.\\
We tried as well hypersurfaces in other homogeneous varieties other than $\Gr(k,n)$, for example using the list of Konno in \cite{konno2}, but none of them satisfied the above condition. For the complete intersections in homogeneous space, we do not have any reasonable conjecture. Atanas Iliev informed us that a FK3 variety can be obtained by taking a 6-codimensional linear section of the $E_6$ variety $\mathbb{O}\PP^2$, but we have not pursued this direction yet.\\
Already in this paper we analyse some extra case that do not fit in our numerological pattern. This is for example the case of $\TT(k,10)$ (and its linear section). However, since this the only reasonable systematic way to produce examples, we decided to write few lines to explain how the list was found and why it stops. To do this, we decided to use as key varieties $Y$ examples that automatically satisfied the Hodge theoretical condition in \ref{num}. 

Let $G$ be one of the varieties below. Consider the positive integer $m$ such that $\omega_G\cong \of_G(-m)$ and $D=\ddim G$. The equations in \ref{num} become
\begin{equation}\label{condition}2t+1=D \textrm{ and } at=m.
\end{equation}
\subsubsection*{$\Gr(k,k+l)$}
For the Grassmannian $\Gr(k,k+l)$ the dimension is $D=lk$ and the index equals $k+l$. First notice that $D$ must be odd.
The equations are $2t+1=kl$ and $at=k+l$, some $a$. Substituting we get $\frac{a(kl-1)}{2}=k+l$ and thus $akl=a+2k+2l$. Since $a\geq 1$ we have $kl \leq a+2k+2l$. It is easy to see that there are no solutions if $k\geq 5$, and for obvious reasons the case $k=2,4$ are excluded. In the case case $k=3$ substituting we get $l=\frac{a+6}{3a-2}$. This implies $a< 3$ for the previous number to be an integer. The case $a=2$ gives an even dimensional Grassmannian, so we discard it. The case $a=1$ corresponds to $G=\Gr(3,10)$. 
The associated FK3 is the Debarre-Voisin variety.
\subsubsection*{$\SGr(k,k+l)$}
The symplectic Grassmannian $\SGr(k,k+l)$ has dimension $kl- {k \choose 2}$ and index equal to $l+1$. If we substitute this in the equation above and look for solutions we find as triple $(k,l,a)=(2,3,2),(3,6,1), (5,3,2), (10,6,1)$. However, if $\omega$ is a non-degenerate skew symmetric $(k+l) \times (k+l)$ matrix, there are no $k$-dimensional isotropic subspaces if $k>l$ and $k+l$ even. We can therefore discard the last two triples and we are left with $X_2 \subset \SGr(2,5)$ (Gushel-Mukai fourfold) and $X_1 \subset \SGr(3,9)$, already considered in
 \cite{ilievmanivel2}.
 \subsubsection*{$\Ml(k,k+l)$}
 The bisiymplectic Grassmannian $\Ml(k,k+l)$ has dimension $kl-k(k+1)$ and index equal to $l-k+2$. If we substitute in the equation above and look for solutions we find as triple $(k,l,a)=(3,5,1), (5,5,1)$. The second one can be identified with a (multi)-linear section of $(\PP^1)^5$, see \cite{kuznetsovpicard}, the first one, an 8-fold linear section of $\Ml(3,8)$ is new.\\
 A similar computation can be done for the tri-symplectic Grassmannian $\mathrm{S}_3\Gr(k,k+l)$. This is relevant since two K3 by Mukai (genus 6 and genus 12) can be considered as (respectively) quadratic and linear section of it. However, no more examples have been found.
\subsubsection*{$\OGr(k,k+l)$}
The orthogonal Grassmannian $\OGr(k,k+l)$ has dimension $kl-{k+1 \choose 2}$ and index $l-1$ (with respect to the Pl\"ucker line bundle $\of_G(1)$, albeit non-irreducible in the Picard group). The only admissible triple is $(3,5,1)$. This is a linear section of the orthogonal Grassmannian $\OGr(3,8)$.
\subsubsection*{$Z_{\mQ^*(1)}$}
This variety is the zero locus of a general global section of the bundle $\mQ^*(1)$ on $\Gr(k, k+l)$. If $k=2$, it is $\TT(k,k+l)$. It has dimension $l(k-1)$ and index $k+1$. There are two admissible triples, $(2,7,1), (6,3,1)$. However the second one can be identified with $X_1 \subset \SGr(3,9)$. The first one $X_1 \subset \TT(2,9)$ is new. Notice that we find as well the generic K3 of genus 4 as $(2,3,3)$ since the zero locus of $Q^*(1)$ on $\Gr(2,5)$ is a quadric threefold. There are as well some FK3 obtained by $\TT(2,n)$. However, they do not fall in this pattern, and we will examine them separately.
\subsubsection*{\textit{Other varieties}}
We tried other bundles  to produce varieties of K3 type, such as $\mR^{\vee}(1)$ or the locus of $\Sym^2 \mR^{\vee} \oplus \W^2 \mR^{\vee}$ (the \emph{orthosymplectic Grassmannian}). Even without no guarantee on the weight of the Hodge structure, our attempt was motivated by some example in the list of K\"uchle, see \cite{kuznetsovpicard}. However, we found no more new example.
\subsubsection*{\textit{Products}}
Products of projective spaces do produce a handful more of examples. One can easily see that no more than 5 projectives can be involved, with the extremal case being $X_{(1^5)} \subset (\PP^1)^5$. Other examples are $X_{(1,1,1)} \subset (\PP^3)^3$, and $X_{(2,1,1)} \subset \PP^3 \times \PP^1 \times \PP^1$. In the products of Grassmannians when $k>1$, no further example is found. Indeed the index of a product of Grassmannians has index the gcd$(k_i+l_i)$. Substituting in the equations, one first find that no more than two Grassmannians can be used, and only one of them can have $k>1$. The possible cases are $X_{(1,1)} \subset \Gr(2,6) \times \PP^5$ and $X_{(2,1)} \subset \Gr(2,4) \times \PP^1$. Identical computations yield all the remaining cases.

\subsection{Geometric tools and tricks}
\subsubsection{A blow-up lemma}
We state here a blow up lemma. Although it merely descends from definitions, it is worth to recall it. It is worth to point out that a similar lemma is used in \cite{kuznetsovpicard}.
\begin{lemma}\label{lem:blowup} Let  $X=X_{(d, 1)} \subset Z \times \PP^1$. Then $X \cong Bl_S Z,$ where $S$ is the intersections of $2$ divisors of degree $d$ on $ Z$.
\end{lemma}
\begin{proof}
Let $\PP^1=\Proj(\C[y_0,y_1])$ and $V^{\vee} \cong \C[y_0,y_1]_1$. (that is, homogeneous forms of degree 1). Denote by $W^{\vee} \cong H^0(\of_{Z}(d)) $. $X$ is given by definition by a choice of $\lambda \in  W^{\vee} \otimes V^{\vee}$, or equivalently by a map (that we will still denote by $\lambda$) $\lambda: V \longrightarrow W^{\vee}.$
This map gives a 2-dimensional subspace of $ W^{\vee}$, or equivalently a pencil of divisors in $ Z$. The base locus of this pencil coincides with the $S$ defined in the lemma. The (only) incidence equation for the blow up of $Z$ in $S$ is $y_0f_d+y_1g_d$ and this is of course the same equation defining $X$. This proof admits an obvious generalisation when $\rho(Z) >1$.
\end{proof}
\subsubsection{Higher codimension case and Cayley trick(s)}
The above blow-up lemma admits a higher-codimensional generalisation. Indeed, when $X$ is the zero locus of a $(1,1)$ divisor in $U \times \PP^{r-1}$ (with the obvious generalisation if $\rho(U) >1$) then $X$ can be given either by an element of $ W^{\vee} \otimes V_r^{\vee}$ or as a map $$\lambda: V_r \longrightarrow W^{\vee}.$$ If $r>2$ we cannot identify $X$ with any birational modification of the pair $(U,S)$, where $S$ is the base locus of the above linear system. However $X$ and $S$ share a deep relation, known as the \emph{Cayley trick}. More precisely the result is the following
\begin{thm}[Thm. 2.10 in \cite{orlov}, Thm 2.4 in \cite{kimkim}] \label{cayley}  
Let $q:E \rightarrow U$ be a vector bundle of rank $r\ge 2$ over a smooth projective variety $U$
and let $S=s^{-1}(0)\subset U$ denote the zero locus of a regular section $s \in H^0(U,E)$ such that $ \dim S = \dim U - \mathrm{rank}\, E$. 
Let $X=w^{-1}(0) \subset \PP E^\vee$ be the zero locus of the section $w\in H^0(\PP E^\vee, \cO_{\PP E^\vee}(1))$  
determined by $s$ under the natural isomorphism
$H^0(U,E)\cong H^0(\PP E^\vee, \cO_{\PP E^\vee}(1))$. 
Then we have the semiorthogonal decomposition
$$ D^b(X)= \langle q^*D^b(U), \cdots, q^*D^b(U) \otimes_{\cO_X} {\cO_X}(r-2), D^b(S) \rangle .$$
\end{thm}
When this happens, we will write $D^b(S) \hookrightarrow D^b(X)$. There is as well an (older) analogue Hodge-theoretic statement, cf. Prop. 4.3 in \cite{konno}, stating that the vanishing cohomologies of $S$ and $X$ are isomorphic up to a shift. When the hypotheses of the above Theorem are verified, this therefore proves at once that $X$ is of K3-type.\\
The Cayley trick can be generalised in the following way, using the formalism of Homological projective duality.
\begin{proposition} \label{rennemo}  Let $Y_1$ and $Y_2$ be a pair of varieties with Lefschetz decompositions and embedded in $\PP(V)$. Let $Z_H$ be the intersection of $Y_1 \times Y_2$ with a general (1,1)-divisor $H$. Let $f_H$ be the map that $H$ naturally defines from $\PP(V)$ to $\PP(V^\vee)$. Let $X_H = Y_1 \cap f_H^{-1}(Y_2^\vee$), where $Y_2^\vee$ is the Homological Projective dual to $Y_2$. Then $D(X_H) \hookrightarrow D(Z_H)$.
\end{proposition}
\begin{proof}
Let $D(Y_2)=\langle A_0,A_1(1),\dots A_{m}(m) \rangle$ be the given Lefschetz decomposition of $Y_2$.
The divisor $H$ parametrizes, for every point of $Y_1$, an hyperplane section of $Y_2$, hence it defines a map $f_H\,:Y_1\,\rightarrow\,\PP(V^\vee)$. In this way, $Z_H$ is identified with the pullback through $f_H$ of the universal hyperplane section $\mathcal{Y}_2\subset Y_2\times \PP(V^\vee)$. Now, by \cite[Lemma 3.3]{kuzi_sod2} we have
$$D(\mathcal{Y}_2)=\langle D(Y_2^\vee), A_1(1)\boxtimes D(\PP(V^\vee)),\dots, A_m(m)\boxtimes D(\PP(V^\vee)) \rangle.$$
By applying base change \cite[Thm 5.6]{kuzi_sod} to the diagram 
$$\xymatrix{ Z_H  \ar[d]^\iota \ar[r]
& \mathcal{Y}_2\ar[d]^{\pi_2} \\ Y_1\ar[r]^{f_H} & \PP(V^\vee), } $$
we obtain:
$$D(Z_H)=\langle D(Y^\vee_2 \times_{\PP(V^\vee)} Y_1), A_1(1)\boxtimes D(Y_1),\dots, A_m(m)\boxtimes D(Y_1)\rangle.$$
And the variety in the first factor here is precisely $X_H=Y_2^\vee \times_{\PP(V^\vee)} Y_1=Y_1 \cap f_H^{-1} (Y_2)^\vee.$
\end{proof}

\section{Case-by-case analysis}

\subsection{Identifications}
Before analysing in details the examples in our list, we want to eliminate some varieties that are well-known examples in disguise. We recall some results of Kuznetsov, that we conveniently bundle together. Recall that the variety $\Ml(k,n)$ is the \emph{bisymplectic Grassmannian}. It can be thought either as the intersection of two symplectic Grassmannian $\SGr(k,n)$ inside $\Gr(k,n)$ or as the zero locus over $\Gr(k,n)$ of a general global section of the bundle $\W^2 \mR^{\vee} \oplus \W^2 \mR^{\vee}$. We will better describe this variety later in the paper.
\begin{thm}[Thm 3.1 and Cor. 3.5 in \cite{kuznetsovpicard}]\label{kuzzolo}The following hold:
\begin{itemize} 
\item There is an isomorphism  $\Ml(n,2n) \cong \prod (\PP^1)^n$;
\item The variety $X_{(1,1,1,1,1)} \subset \prod (\PP^1)^5$ is isomorphic to $W=Bl_S(\prod (\PP^1)^4)$, where $S=S_{(1,1,1,1)^2}$ is a non-generic K3 surface of genus $g=13$, given as the intersection of two divisors of multidegree $(1,1,1,1)$.
\end{itemize}
\end{thm}
Some of the Fano of K3 type that we found in our search can be actually identified with the $W$ above. For this reason they are not included in our main table. More precisely we have
\begin{lemma} Let $W$ the Fano of K3 type in \cite{kuznetsovpicard} defined above. Then the following Fano of K3 type \begin{itemize}
\item $X_{(1,1,1,1,1)} \subset Q_2 \times Q_2 \times \PP^1$;
\item $X_{(1,1,1,1,1)} \subset \Ml(4,8) \times \PP^1;$
\item $X_{(1,1,1,1,1)} \subset \Ml(3,6) \times \Ml(2,4);$
\end{itemize}  
are isomorphic to $W$.
\end{lemma}
\begin{proof}
The first case is obvious, since $Q_2 \cong \PP^1 \times \PP^1$. For the other two cases, by definition and Kuznetsov's result $\Ml(n,2n)$ coincides with $\prod (\PP^1)^n$. 
\end{proof}
There is one more identification between two numerological candidates.
\begin{lemma} $X_{(1,1,1)} \subset \mathbb{S}_3 \times \PP^1 \times \PP^1 \cong X_{(2,1,1)} \subset \PP^3 \times \PP^1 \times \PP^1$.
\end{lemma}
\begin{proof} It follows from the well known identification $ \mathbb{S}_3 \cong \PP^3$, see for example \cite{kuznetsovs}. The difference in the degree is explained since the line bundle giving the spinor embedding for $\mathbb{S}_3 $ is the square root of the Pl\"ucker one.
\end{proof}

\subsection{Blow-up and Mukai type}
To prove that each of the variety of type M and B are of K3 type one can use the Cayley trick statement, as in Theorem \ref{cayley}. Indeed the (stronger) derived category statement implies the Hodge theoretical one. Indeed this can be seen by writing down such a semiorthogonal decomposition as prescribed by \ref{cayley} and then taking Hochshild homology. Alternatively one can use Riemann-Roch and standard exact sequences to compute the relevant Hodge numbers. We did these calculations as sanity checks for all our examples, however we believe it is neither worth nor interesting to list all of them, since they are quite similar. Therefore we will include just one example, namely Proposition \ref{gq1}, where Theorem \ref{cayley} does not apply in a straightforward way. For the families B1 and B2, Lemma \ref{lem:blowup} settles the matter. \\
In terms of construction of polarised families of IHS, we investigate another construction of the Hilbert scheme of points on a genus 8 K3, see Proposition \ref{gp2}. We believe that each of the examples in our list of Fano could lead to similar constructions: this would be especially interesting, considering the lack of examples of polarised families of Hilbert schemes of points on K3 surfaces.

\subsection{M3: a (different) computation in intersection theory} The variety M3  is $X_{(1,1)} \subset \Gr(2,5) \times Q_5$. It has dimension 10 and index 4. It is neither a blow up with a center in a K3 surface, nor we can apply the Cayley trick. However we can show that it is a Fano of K3 type using Proposition \ref{rennemo}. Indeed we have
\begin{lemma} Let $S_6$ be a K3 surface of genus 6 in the Mukai model and $X$ our M3 as defined in the table. Then $D^b(S_6) \hookrightarrow D^b(X)$.
\end{lemma}
\begin{proof} It suffices to apply Proposition \ref{rennemo}, since the Grassmannian $\Gr(2,5)$ (or even a quadric hypersurface) is projectively self-dual. The intersection of the Grassmannian $\Gr(2,5)$ with a 5-dimensional quadric (or, equivalently, the intersection of $\Gr(2,5)$ with a quadric and 3 hyperplanes in its Pl\"ucker embedding) is a K3 of genus 6 and degree 10 by Mukai's classification. To conclude one needs to argue that the orthogonal complement to the derived category of $D^b(S_6)$ in $D^b(X)$ is generated by an exceptional collection, and then taking Hochshild homology (which is additive on semi-orthogonal decompositions), together with the Hochshild-Konstant-Rosenberg isomorphism cf. \cite[Theorem 7.5, 8.3]{kuzhkr}.
\end{proof} 

As an alternative methods we can show that M3 is of K3 type using a lengthy (but rather standard) play with long exact sequences and cohomological vanishings.
\begin{proposition}\label{gq1}Let $X=X_{(1,1)} \subset \Gr(2,5) \times Q_5$. Then $X$ is of K3 type.
\end{proposition}
The proof of the above proposition can be split in two lemma. The first one is a Chern class computation, the second one is essentially an application of Bott's theorem.
\begin{lemma} \label{erchar}The topological Euler characteristic of $X$ is $e(X)=72$.
\end{lemma}
\begin{proof}This is a lengthy (but direct) exercise in intersection theory, and we will spare the details to the reader. Let us denote $Y=\Gr(2,5) \times Q_5$. Denote by $\alpha_1=c_1(\of_Q(-1))$ and $\beta_1=c_1(\of_G(-1))$. Denote by $\beta_2=c_2(\mR)$. One has $H^4(\Gr(2,5), \Z) = \langle \beta_1^2, \beta_2 \rangle$. One easily compute $c(Q)$, $c(G)$ and $c(Y)=c(G)c(Q)$. In particular by Gauss-Bonnet $c_{11}(Y)=-6 \alpha_1^5 \beta_1^6$ and $$e(Y)= \int_{Y} -6 \alpha_1^5 \beta_1^6=60.$$
We then use the normal sequence associated to $X$ $$ 0 \to T_X \to TY|_X \to \of_X(1,1) \to 0.$$
This implies $c(TY|_X)=c(X)(1-\alpha_1-\beta_1)$. We can compute recursively the Chern classes of $X$, with in particular $$c_{10}(X)=(9\alpha_1^5\beta_1\beta_2^2+9\alpha_1^4\beta_1^2\beta_2^2)|_X.$$ To compute the restriction we evaluate against the class of $X$, and we have $ c_{10}(X) \cdot X=18\alpha_1^5\beta_1^2\beta_2^2.$
Using the relation in $A(G)$ given by $2\beta_1^5=5\beta_1\beta_2^2$ we get $$c_{10}(X) \cdot X= \frac{2 \cdot 18}{5}\alpha_1^5 \beta_1^6=\frac{6}{5}e(Y)=72.$$
\end{proof}

\begin{lemma} For $0 \leq i \leq 3$  we have $h^{i, 10-i}(X)=0$. Moreover $h^{6,4}(X)=h^{4,6}(X)=1$.
\end{lemma}
\begin{proof} As before let us denote $Y=\Gr(2,5) \times Q_5$, and with $\mL \cong \of_Y(1,1)$ (and its restriction to $X$ as well). We use the following two exact sequences \begin{equation} \label{seq1}
0 \to \Omega_X^{k-1} \otimes \mL^{\vee} \to \Omega^k_{Y|X} \to \Omega^k_X \to 0 \end{equation} and 
\begin{equation}\label{seq2} 0 \to \Omega^k_Y \otimes \mL^{\vee} \to \Omega^k_Y \to \Omega^k_{Y|X}\to 0,
\end{equation}
possibly twisting for some positive multiple of $\mL^{\vee}$ when required. The computation is rather lengthy and technical, and we will skip most of the details. To find similar computations the reader can refer to \cite{eg1}. For the results on the cohomological vanishings for both $\Gr(2,5)$ and $Q_5$ one can consult for example \cite{peternell}, \cite{snow}.\\
The first vanishing $h^{0,10}(X)$ is obvious. Let us show the first non-obvious one, that is $h^{1,9}(X)=0$. Consider the two sequences \ref{seq1} and \ref{seq2} above with $k=1$. Using the K\" unneth formula one easily see that the cohomology of $\Gr(2,5) \times Q_5$ is of Lefschetz-type. Moreover from Kodaira vanishing and since $H^{10}(X, \mL) \cong H^{0}(X, \of_X(-3,-3))=0$ one reduces to
$$ 0 \to  H^9(\Omega^1_{Y|X}) \to H^9 (\Omega^1_X) \to 0 $$ and $$ 0 \to H^9(\Omega^1_{Y|X}) \to H^{10}(\Omega^1_Y \otimes \mL^{\vee}) \to 0.$$
However, if we denote with $\pi_1$ (resp. $\pi_2$) the projection on $\Gr(2,5)$ (resp. $Q_5$) we have $\Omega^1_Y \cong \pi_1^* \Omega^1_{\Gr(2,5)} \oplus \pi_2^* \Omega^1_Q$, and from the K\"unneth formula for the box product and the well known vanishings for the twisted cohomologies of $\Gr(2,5)$ and $Q_5$ we have $$H^{10}(\Omega^1_Y \otimes \mL^{\vee})\cong H^9(\Omega^1_{Y|X}) \cong  H^9 (\Omega^1_X)=0.$$
For $h^{2,8}(X)$ we use the sequences \ref{seq1} and \ref{seq2} with $k=2$ and $k=1$ twisted by $\mL^{\vee}$. Indeed one has from \ref{seq1}
$$ 0 \to H^8(\Omega^2_{Y|_X}) \to H^8(\Omega^2_X) \to H^9(\Omega^1_X)  \to H^8(\Omega^2_{Y|_X}) \to 0.$$ The two external terms can be checked to be 0 using \ref{seq2}, again together with the K\"unneth formula and the usual vanishings (using the decomposition for $\Omega^2_Y$). Using the twisted version of \ref{seq1} and \ref{seq2} we reduce to the isomorphism $H^8(\Omega^2_X) \cong H^{10}((\mL_X^{\vee})^{\otimes 2})=0$. The same argument works as well for $h^{3,7}(X)=0$, where for $h^{4,6}(X)$ we get $$H^6(\Omega^4_X) \cong H^{10}((\mL_X^{\vee})^{\otimes 4})\cong H^0(\of_X) \cong \C.$$
\end{proof}

The last Lemma is enough to prove that $X$ is of K3 type. In particular, when combined with Lemma \ref{erchar} we explicitely compute all the Hodge numbers. The following corollary is in fact proved bundling the two results above, together with Lefschetz theorem on hyperplane section and a direct application of the K\"unneth formula.
\begin{corollary} Suppose $p+q \neq 10$. The only non-zero Hodge numbers $h^{p,q}$ of $X$ are $$h^{0,0}=h^{10,10}=1, \ h^{1,1}=h^{9,9}=2, \ h^{2,2}=h^{8,8}=4, \ h^{3,3}=h^{7,7}=6, \ h^{4,4}=h^{6,6}=8.$$
For $p+q=n$ the only non-zero Hodge numbers are $$h^{6,4}=h^{4,6}=1, \ h^{5,5}=28,$$ with moreover the dimension of the vanishing cohomology subspace $h^{5,5}_{\van}=19$.
\end{corollary}

\subsection{M7: another construction of $S_8^{[2]}$}
The 12-fold $X_{M7}$ is given by the zero locus of a (1,1) section on $\Gr(2,6) \times \PP^5$. Let $S_8=\Gr(2,6) \cap H_1 \cap \ldots \cap H_6$. Then $S_8$ is a general K3 surface of genus 8 in Mukai's model. From the Cayley trick argument one has that $D^b(S_8) \hookrightarrow D^b(X_{M7})$. On the Hodge-theoretical level indeed we have:
\begin{lemma}\label{m7} Let $X_{M7}$ as above. Then  $X_{M7}$ is of K3 type with $h^{6,6}=31$ and the vanishing subspace $h_{\van}^{6,6}=19$.
\end{lemma}
\begin{proof} Since  $\Gr(2,6) \times \PP^5$ is a central variety, it is enough to compute the Euler characteristics $\chi(\Omega^i)$ for $i=5,6$. This can be done for example via Riemann-Roch or using Macaulay2.
\end{proof}
As expected, we can associate to $X_{M7}$ an IHS, which is linked to the genus 8 K3. To do this, let $Z$ be given by the zero locus of a general global section of the bundle $\W^2 \mR^{\vee}_{4,6} \otimes \mR^{\vee}_{2,6}$ on $\Gr(4,6) \times \Gr(2,6)$. We have the following proposition.
\begin{proposition} $Z$ is an IHS fourfold.
\end{proposition}
\begin{proof}Recall the formula for the first Chern class  of a product $c_1(\W^2 \mR^{\vee}_{4,6} \otimes \mR^{\vee}_{2,6})= \mathrm{rk}(\mR^{\vee}_{2,6})\cdot c_1(\W^2 \mR^{\vee}_{4,6} )+\mathrm{rk}(\W^2 \mR^{\vee}_{4,6})\cdot c_1( \mR^{\vee}_{2,6}).$ By adjunction it follows that for a general section $Z$ is a smooth fourfold with $c_1=0$. We compute now its holomorphic Euler characteristic $\chi(\of_Z)$.
This can be done for example via a Riemann-Roch computation, since $$\chi(\of_Z) = \frac{c_2^2-c_4}{720}.$$ We will use a Macaulay2 code in order to speed up the calculation.
\begin{verbatim}
loadPackage "Schubert2"
k1=2, l1=4, k2=4, l2=2;
G26=flagBundle({k1,l1}, VariableNames=>{r1,q1});
(R1,Q1)=G26.Bundles;
V=abstractSheaf(G26, Rank=>6);
G46=flagBundle({k2,l2}, V, VariableNames=>{r2,q2});
(R2,Q2)=G46.Bundles;
p=G46.StructureMap;
R1G46=p^*(dual R1);
F=R1G46**exteriorPower_2 dual R2;
Z=sectionZeroLocus F;
chi(OO_Z);
\end{verbatim}
Running the previous code one verifies $\chi(\of_Z)=3$. in particular the statement follows by simply applying Beauville-Bogomolov decomposition theorem.
\end{proof}
The deformation type of $Z$ can be shown to be the expected one as follows.
\begin{proposition}\label{gp2}  $Z$ is isomorphic to Hilb$^2(S^8)$.
\end{proposition}
\begin{proof} Let $h \in \W^2V_6^* \otimes V_6^*$ defining $X_{M_7}$. As above, we can consider $h$ as a morphism $$h: V_6 \to \W^2 V_6^*.$$ A point in Hilb$^2(S_8)$ is therefore given by a pair $(u_1, u_2)$, $u_i \in \W^2 V_6$  on both of which $h$ vanishes. Consider $W \subset \W^2 V_6$ spanned by $u_1, u_2$. Consider further the restricted morphism $\overline{h}^t: W \to V^\vee_6$. This has rank 2, and we can take $P= \mathrm{Im}(\overline{h}^t)$. By construction $h$ vanishes on the pair $(W,P) \in \Gr(4,6) \times \Gr(2,6)$, thus defining a point in $Z$. From this construction, it is clear that $W$ determines $P$. Moreover, the map we constructed inside $\Gr(4,6)$ can be seen as the same map (after duality) which associates to Hilb$^2(S_8)$ a line in the pfaffian cubic fourfold, hence it is an isomorphism.  \end{proof}
We point out the similarities between this contruction and \cite[Proposition B.6.3]{kps}. Here it is proved how the variety of lines (resp. conics) of a smooth cubic threefold (resp. a generic Fano threefold of genus 8) is isomorphic to a section of the bundle $\W^2 \mR^{\vee}_{4,5} \otimes \mR^{\vee}_{2,5}$ over $\Gr(4,5) \times \Gr(2,5)$. In turn, their proof can be modified to give an alternative proof of \ref{gp2}.
\subsection{Sporadic examples}
This subset of the list is the most interesting one. Indeed for these Fano we cannot produce a systematic method as in the \emph{Mukai} case. For each one of them already proving that they are of K3 type requires an ad-hoc strategy. Our most interesting results comes indeed from this section: indeed we reinterprete the Debarre-Voisin IHS fourfold as moduli space of relevant objects on a Fano of K3 type in two different ways, namely as in Theorem \ref{hk} and Proposition \ref{t2}. Moreover we produce the first examples of a Fano with multiple K3 structures, cf. Proposition \ref{3k3} and with a mixed $(2,3)$ CY structure, cf. Proposition \ref{23cy}. Moreover we do not limit ourselves to the computation of the Hodge numbers: we give indeed geometrical descriptions of many of the examples we consider, since we believe them to be a rich and beautiful sources of geometries.
\subsection{S1: four codimensional linear section of $\Gr(2,8)$} 
We already considered this example in our previous work \cite{eg1}, therefore we will not spend too much time on it. It is described in a surprisingly simple way as a codimensional 4 linear section of the Grassmannian $\Gr(2,8)$.
\begin{proposition}\label{s1} Let $X_{1,1,1,1} \subset \Gr(2,8)$ given by a generic section of $\of_G(1)^{\oplus 4} $. Then $X$ is an 8-fold of K3 type,with $h^{4,4}_{\van}(X)=19$.
\end{proposition}
We remark that there is another FK3 closely related to S1. This is $X_{(1,1)} \subset \Gr(2,8) \times \PP^3$. In our main table this is listed as M13. We chose this notation since, although there is no K3 in the Mukai model related, it shares many similarities with the other Fano in the \textbf{M} group. In particular one can apply directly Cayley trick to prove that this Fano is of K3 type.\\
 As already remarked in our previous work the projective dual of $X_{1,1,1,1} \subset \Gr(2,8)$ is quartic K3 surface $S_3 \subset \PP^3$. An embedding of the derived category of the quartic K3 inside the derived category of the above linear section is proved in \cite{segalthomas}, Thm 2.8.\\
We already conjectured that this complete intersection in $\Gr(k,n)$ should be the only FK3 obtained in this way. We repeat the precise formulation of this conjecture here.
\begin{conj} Let $X=X_{d_1, \ldots, d_c} \subset \Gr(k,n)$ a Fano smooth complete intersection of even dimension. Then $X$ is not of K3-type unless $$(\lbrace d_i \rbrace, k,n)=(\lbrace 3 \rbrace,1,6),(\lbrace 2,1\rbrace,2,5), (\lbrace 1,1,1,1\rbrace,2,8), (\lbrace 1 \rbrace, 3,10).$$ 
\end{conj}

\subsection{S2: a K3 of genus 7 from $\OGr(3,8)$}
This sporadic example is a linear section $X= \OGr(3,8) \cap H$ of the orthogonal Grassmannian $\OGr(3,8)$. It is worth to spend few words on the ambient variety. In general the orthogonal Grassmannian $\OGr(n-1, 2n)$ behaves differently from $\OGr(k,2n)$, which for $k \neq n-1$ is a prime Fano variety. Indeed $\OGr(n-1, 2n)$ can be realised as a $\PP^{n-1}$ bundle over (both of) $\mathbb{S}_{n}^i$, the latter denoting the two connected component of the maximal orthogonal Grassmannian $\OGr(n, 2n)$ in the spinor embedding. In particular the Picard group of  $\OGr(n-1, 2n)$ has rank 2 with the Pl\"ucker line bundle $\mathcal{L} :=  \of_{\mathbb{S}^1} (1)\boxtimes  \of_{\mathbb{S}^2}(1)$ is very ample.  $\OGr(n-1, V_{2n})$ is non-degenerate in the Pl\"ucker embedding, and $$H^0(\OGr(n-1, 2n), \mathcal{L}) \cong \W^{n-1} V_{2n}^{\vee}.$$
With $X=X_1 \subset \OGr(3,8)$ in the introductory table we mean the zero locus of a generic global section of $\mathcal{L}$. Such $X$ is an 8-fold of index $\iota=3$. Since it is a linear section of a central variety, to compute its Hodge numbers it suffices to compute the Euler characteristics $\chi(\Omega^i_X)$, together with the knowledge of the cohomology of $\OGr(3,8)$. A full computation by the means of Borel-Bott-Weil theorem, can be found in the PhD thesis of the first author. We recall here the result. 
\begin{lemma}[cf. \cite{thesis}, Proposition A.1.1] $X$ is a Fano 8-fold of K3 type with $h^{4,4}(X)=24$, and its vanishing subspace of rank 19.
\end{lemma}
We explain now a link between this 8-fold $X$ and a K3 of genus 7. Recall from the work of Mukai that a generic K3 of such genus can be obtained by cutting $\mathbb{S}_{10}$ with 8 hyperplanes. Here we use a different description of the aforementioned K3. Let $X \subset \OGr(3,8)$ defined by $V(\sigma)$, $\sigma \in H^0(\mL)$. Let $\mathbb{S}_8$ be (one of the two connected component of) the orthogonal Grassmannian $\OGr(4,8)$, denote with $\mR$ the restriction of the tautological bundle. Since $\sigma$ can be seen as an element in $H^0(\mathbb{S}_8, \W^3\mR^{\vee})$ we can denote by $S= V(\sigma) \subset \mathbb{S}_8$. It is easy to check that $S$ is a K3 of genus 7 (notice that $\mathbb{S}_8$ is nothing but a 6-dimensional quadric hypersurface in disguise, either using triality or checking dimension and invariants). Such $S$ is responsible for the interesting part of the derived category (and therefore the Hodge theory of $X$). Indeed we quote the following result of Ito-Miura-Okawa-Ueda. Denote $\pi$ the restriction of the projection $p$ from $X$ to (one of the two) $\mathbb{S}_8$.
\begin{lemma}[Lemma 2.1 in \cite{ito}]  The morphism $\pi$ is a $\PP^2$-bundle over $\mathbb{S}_8 \smallsetminus S$ and a $\PP^3$-bundle over $S$, locally trivial in the Zariski topology.
\end{lemma}
In turn we can use an adapted version of Orlov's blow-up formula to this case. This is indeed a generalisation of the Cayley trick. We borrow this result from the forthcoming \cite{nested}, where it will be shown in full details and generality. For this reason, the proof will be omitted here.\\
First, in the notation above, denote by $\iota: S \subset \mathbb{S}_8 $. The above Lemma is equivalent to the following commutative diagram


$$\xymatrix{
F \ar@{^{(}->}[r]^j \ar[d]_p & X \ar[d]^\pi \\
S \ar@{^{(}->}[r]^\iota & \mathbb{S}_8,}$$

with $F$ a smooth projective subvariety, $j:F \subset X$ of codimension $d=4+2-3=3$ and a locally free sheaf $\mathcal{F}$ of rank $4$ on $S$ such that $p:F \simeq \PP_S(\mathcal{F}) \to S$.
We denote by $\ko_F(H)$ the relative ample bundle of $p$ and we assume that there is a line
bundle $\ko_Y(H)$ such that $\ko_Y(H)_{\vert F} \simeq \ko_F(H)$ and that there is a vector
bundle $\ke$ of rank $d$ on $X$ such that $F$ is the zero locus of a general section
of $\pi^*\ke \otimes \ko_Y(-H)$.

%
%

We define the functors $\Phi_l:\Db(S) \to \Db(F)$ by the formula
$\Phi_l(A)= j_* (p^* A \otimes \ko(lH))$.

\begin{proposition}\label{o2}
In the configuration above, $\Phi_l$ is fully faithful for any integer $l$, and there is a semiorthogonal
decomposition:
$$\Db(Y)=\sod{\Phi_{-1}\Db(Z),\pi^*\Db(X),\ldots,\pi^*\Db(X)\otimes \ko_Y(2H)}$$
\end{proposition}

\subsection{S3: bisymplectic Grassmannian $\Ml(3,8)$ and Debarre-Voisin IHS} The variety $\Ml(k,n)$ is given by the vanishing of a global section of the bundle $\W^2 \mR^{\vee} \oplus \W^2 \mR^{\vee}$ on the Grassmannian $\Gr(k,n)$. Equivalently, given a pencil $\lambda: \C^2 \to \W^2 V_n^{\vee}$ it parametrises k-dimensional subspaces isotropic for all skew-forms in the pencil. In \cite{kuznetsovpicard} this variety is studied by Kuznetsov when $k=n/2$ and by Benedetti in \cite{ben18} with a strong emphasis in the case $k=2$.  Let us recall some key facts of the construction.  Assume that $n=2m$ is of even dimension. To a general pencil $\lambda$ are canonically associated $m$ degenerate skew-forms $\lbrace \lambda_1, \ldots, \lambda_m \rbrace$, given by the intersection bewteen the line $L_{\lambda} \subset \PP(\W^2 V^{\vee})$ and the (Pfaffian) discriminant hypersurface $D$, corresponding to degenerate skew-forms. Denote by $K_i$ the kernel of $\lambda_i$. The smoothness of $\Ml$ is equivalent to the $\lambda_i$ being distinct, and moreover we can decompose $V=K_1 \oplus \ldots \oplus K_m$ as a direct sum.\\
Kuznetsov gives as well the canonical form for the pencil, espressing the two generators (up to dividing by 2) as $$\omega_1= x_{1,2}+x_{3,4}+\ldots +x_{n-1,n}, \ \ \omega_2=  a_1x_{1,2}+a_2x_{3,4}+\ldots +a_m x_{n-1,n},$$ with the $a_i$ pairwise distinct, and $x_{i,j}:= x_i \wedge x_j$. This way, we can identify $K_1:=\langle e_1, e_2 \rangle$, $K_2:=\langle e_3, e_4 \rangle$ and so on. When $m=k$ one has $\Ml(k, 2k) \cong \prod (\PP^1)^k$, see the theorem already recalled in \ref{kuzzolo}. When $m\neq k$ however we do not have such a nice description as a product. For $k=2$ for example $\Ml(2,n)$ is an intersection of $\Gr(2,n)$ with a linear subspace of codimension 2. \\
Let us now focus on the case $\Ml(3,8)$. We compute first the cohomology of a linear section of $\Ml(3,8)$.

\begin{proposition} \label{ah}A linear section $X_1=V(\sigma_1) \subset \Ml(3,8)$ is of K3 type.
\end{proposition}
\begin{proof} 
The first thing to prove is that $\Ml(3,8)$ is a central variety. This can be done via a direct computation, for example using Borel-Bott-Weil theorem. There is however another (much easier) argument which is however more conceptual, and prove the similar statement for all $\Ml(k,n)$. Indeed in \cite[Proposition 2.10]{ben18} it is proved that there is a torus $T \cong (\C^*)^n$ acting on $\Ml(k,n)$ with the fixed locus constituted only by $2^k {n \choose k}$ points. This implies, thanks to \cite[Theorem 2]{sommese} that the $\Ml(k,n)$ is a central variety, with $2^k {n \choose k}$ being its topological Euler characteristic.\\
Lefschetz theorem on hyperplane section enables us to describe the cohomology of $Z$ except all the Hodge groups $h^{p,q}(Z)$ with $p+q=8$. We can determine these dimensions by computing the Euler characteristics of $\chi(\Omega_X^i)$. The latter can be computed via a direct but lengthy computation, and computer algebra systems as Macaulay2 can speed up everything. One has in particular \begin{align*}&\chi(\Omega^1_X)=\chi(\Omega^1_{\Ml(3,8)})=1\\
&\chi(\Omega^2_X)=\chi(\Omega^2_{\Ml(3,8)})=2\\
&\chi(\Omega^3_X)=\chi(\Omega^3_{\Ml(3,8)})+1=7\\
&\chi(\Omega^4_X)=26.
\end{align*}
\end{proof}
This gives as well all the Hodge numbers. We collect them in the next corollary for the reader's convenience.
\begin{corollary}The only non-zero Hodge numbers $h^{p,q}$ of $\Ml(3,8)$ are $$h^{0,0}=h^{9,9}=1, \ h^{1,1}=h^{8,8}=1, \ h^{2,2}=h^{7,7}=2, \ h^{3,3}=h^{6,6}=6, \ h^{4,4}=h^{5,5}=6.$$
\end{corollary}

\begin{corollary} \label{hodges2}Suppose $p+q \neq 8$. The only non-zero Hodge numbers $h^{p,q}$ of $X$ are $$h^{0,0}=h^{8,8}=1, \ h^{1,1}=h^{7,7}=1, \ h^{2,2}=h^{6,6}=2, \ h^{3,3}=h^{5,5}=6.$$
For $p+q=8$ the only non-zero Hodge numbers are $$h^{3,5}=h^{5,3}=1, \ h^{4,4}=26,$$ with moreover $h^{5,5}_{\van}=20$.
\end{corollary} 
We want now to associate to our Fano of K3 type $X$ an IHS $Z$. To do this, at first notice that $\Ml(3,8)$ is degenerate in the Pl\"ucker embedding in $\PP(\W^3 V_8)$. It lies indeed in $\PP(U)$, where 
$$U:= \ker (\varphi: \W^3 V_8 \stackrel{(\contr_1, \contr_2)}{\longrightarrow} V_8 \oplus V_8),$$ where $\contr_i$ denotes the contraction with the 2-skew form $\omega_i$. Equivalently, we have that $\Ml(3,8)$ is defined by a general  $\sigma_1 \in U^{\vee}$.\\
Consider now the Grassmannian $\Gr(6,8)$. Denote by $\bar{\contr_i}$ the contraction with the restriction of the two form $\omega_i|_W$ to a 6-space $W$. For the generic $[W] \in \Gr(6,8)$ the map $$\overline{\varphi}: \W^3 W \stackrel{(\bar{\contr_1}, \bar{\contr_2})}{\longrightarrow} W \oplus W$$ remains surjective, since the rank of  $\omega_1|_W$ and $\omega_2|_W$ is still maximal. 
However, when $W$ is such that every element of the pencil restricted to such $W$ has rank 4, then the above map is not surjective anymore. 
As a special example, one can take a subspace given by $x_1=x_3=0$. Then for example the vector $(e_5, d e_5)$, $d \neq 1$ is not in the image of $\overline{\varphi}$.
To identify in general the locus $D$ where $\overline{\varphi}$ is not surjective let us write (in the notation above) $V_8=K_1 \oplus K_2 \oplus K_3 \oplus K_4$. We can then describe $D$ as $$D:= \lbrace W_6 \subset V_8 \ | \ \ddim (W_6 \cap K_i) \geq 1, \ \forall i \rbrace.$$
$D$ is therefore isomorphic to a $\Gr(2,4)$-bundle over $(\PP^1)^4 \cong \Ml(4,8)$, and over it we have a cokernel sheaf $\mathcal{G}$ of rank 4 on its support, given by the Kernel of the rank 4 map $W \to W^* \oplus W^*$. Summing up, we have the following result
\begin{proposition} On $\Gr(6,8)$ there is an exact sequence of sheaves
$$ 0 \to F \to \W^3 \mR \to \mR \oplus \mR \to \mathcal{G} \to 0.$$
\end{proposition}
\begin{corollary} $F^\vee$ is a globally generated vector bundle of rank 8 and $H^0(F^{\vee})= U^{\vee}$.
\end{corollary}
\begin{proof}
Dually, there is a surjective morphism of sheaves $\W^3 \mR^\vee \to F^\vee$, which is surjective on stalks. Hence, global sections of $F^\vee$ which are images of global sections of $\W^3 \mR^\vee$ are sufficient to generate stalks, so that $F^\vee$ is globally generated. 
\end{proof}
Moreover, since $\mathcal{G}$ is a torsion sheaf supported in codimension 4 we have the following corollary.
\begin{corollary} $c_1(F^{\vee})=8h.$
\end{corollary}
\begin{proposition} Let $Z \subset \Gr(6,8)$ defined by the zero locus of a general global section of the vector bundle $F^{\vee}$. Then $Z$ is a fourfold with canonical class $\omega_Z \cong \of_Z$.
\end{proposition}
\begin{thm} \label{hk}Let $Z$ as above, and let $Z_{DV} \subset \Gr(6,10)$ the Debarre-Voisin IHS. Then $Z$ is isomorphic to $Z_{DV}$. Moreover, $Z$ can be interpreted as (the compactification of) the space of $\Ml(3,6) \cong (\PP^1)^3$ inside $X_1 \subset \Ml(3,8)$.
\end{thm}
\begin{proof}
With a non canonical choice of a two-space $\langle v,w\rangle = V_2\subset V_{10}$, the three form $\omega$ defining $Z_{DV}$ can be written as $\omega=\omega_8+v^\vee\wedge \sigma_1+w^\vee\wedge \sigma_2$, where $\omega_8$ is a three form on an eight dimensional vector space $V_8$ and $\sigma_i$ are two forms on the same space. The natural projection from $\mathbb{P}(V_{10})$ to $\mathbb{P}(V_8)$ induces a rational map from $\Gr(6,10)$ to $\Gr(6,8)$. For this map, there are three kinds of six-spaces:
\begin{itemize}
\item[Type 0] Six spaces which do not intersect the fixed two space $V_2$.
\item[Type 1] Six spaces meeting the fixed $V_2$ in a line $U_1$.
\item[Type 2] Six spaces containing the fixed $V_2$.
\end{itemize}
By a dimension count and the genericity assumption on $Z_{DV}$, spaces of type 2 do not occur inside $Z_{DV}$. Spaces of type 1 are given by the Schubert cycle $\sigma_{3,0^5}(V_2)$, and inside $Z_{DV}$ this is a curve of degree $132$, which is smooth since one can check that the Schubert cycle we use to obtain it is smooth as well.  The blow up of $Z_{DV}$ along this curve maps into a subvariety of $\Gr(6,8)$ given by six planes where the three form $\omega_8$ is given as the sum of $\sigma_1$ and $\sigma_2$ wedged with the dual of some vectors of the six space itself. This is precisely the variety $Z$ for the forms $\omega_8,\sigma_1,\sigma_2$.
The local picture in the exceptional divisor is given by sending a six plane $U_1\subset U_6$ to the set of all possible six planes in $V_8$ containing $U_6/U_1$, which is a $\mathbb{P}^2$.
The image $\pi(U_6)$ of a six space $U_6\in Z_{DV}$ contains three spaces parametrized by $X_1 \subset \Ml(3,8)$ where the form $\omega_8$ restricts to zero, hence also the two forms $\sigma_1,\sigma_2$ are zero. That is, a point of $Z$ parametrizes a copy $\Ml(3,6)\cong (\mathbb{P}^1)^3$ contained in $X_1$ as claimed above.
    We proved that $Z$ has trivial canonical bundle and, if the rational map we defined above from $Z_{DV}$ has degree one, $Z$ and $Z_{DV}$ would be birational minimal models, hence the map given by the blow-up of $Z_{DV}$ along the curve composed with the projection would be a flop. But a flop is not defined in codimension at most two on an IHS fourfold, hence the map was already well defined and is an isomorphism. Let us prove that this map has indeed degree one: Let $V_6$ and $W_6$ be two points of $Z_{DV}$ with the same projection. Therefore, their basis differ only for multiples of v and w and, after a linear combination, we can suppose that at most two elements differ by these vectors. Let us treat first the case of a single vector: let $V_6=\langle v_1,v_2,v_3,v_4,v_5,v_6 \rangle$  and let $W_6=\langle v_1,v_2,v_3,v_4,v_5,v_6+av+bw \rangle$. As the choice of $V_6$ varies, the coefficients $a,b$ are not constant, hence we can suppose $a=1,b=0$ (which happens in codimension one). Thus on $W$ we have $\omega(v_6+v,x,y)=v\wedge \sigma_1(x,y)$. So, if the six space annihilates such a three form, it must be isotropic for $\sigma_1$, which is clearly impossible on a six space, unless the two form degenerates, which happens in codimension two.\\ On the other hand, if $W_6=\langle v_1,v_2,v_3,v_4,v_5+w,v_6+v \rangle$ we have $\omega(v_6+v,x,y)=v\wedge \sigma_1(x,y)$ and $\omega(v_5+w,x,y)=w\wedge \sigma_2(x,y)$. This implies that the residual four space is isotropic with respect to both forms, which is a codimension twelve condition on the six spaces themselves. Indeed, this is $\Ml(4,8)\cong (\PP^1)^4$ inside $\Gr(4,8)$. Hence, by the genericity assumption on $\omega$, this does not happen in our case. 

\end{proof}

\subsection{S5: a section of a non-central variety} This sporadic Fano of K3 type is rather different from the others. It is a linear section of a certain 7-fold of index 3 that we call $\TT(2,9)$, which is not even central, let alone homogeneous. This 7-fold is the zero locus of a general global section of the bundle $\mQ^*(1)$ on the Grassmannian $\Gr(2,9)$. By Borel-Bott-Weil we interpret $H^0(\Gr(2,9), \mQ^*(1)) \cong \W^3 V_9^{\vee},$ therefore $\TT(2,9)$ (sometimes shortened as $\TT$ in the following proofs) is given by the locus of two-spaces in a 9-dimensonal space which are annihilated by a 3-form. This 7-fold, which is indeed a \emph{congruence of lines} has been considered in the recent work (\cite{faenzi}, Ex. 4.14). As we said, the variety $\TT(2,9)$ is not central, therefore we cannot apply any trick as in Proposition \ref{blowup} to compute the Hodge numbers of its linear section. Therefore we will need to go through a proper Borel-Bott-Weil computation.\\
 We will start by stating the final result on the Hodge numbers. 
\begin{proposition} \label{t129}The Hodge numbers of $\TT(2,9)$ are
\begin{center}
{\small
\[\begin{matrix}
&&&&&&&&1 &&&&&&&&\\
&&&&&&&0&&0&&&&&&&\\
&&&&&&0 &&1&&0&&&&&&\\
 &&&&& 0 && 0 && 0 &&0&&& \\
&&&&0 &&0 && 2 &&0 && 0 &&&&\\
&&&0&&0 & & 2 & &2 && 0 && 0 &&&\\
&&0 && 0 && 0 &&2 &&0 &&0 && 0&&\\
&0 && 0 && 0 && 2&& 2 &&0 &&0 && 0&\\
&&0 && 0 && 0 &&2 &&0 &&0 && 0&&\\
&&&0&&0 & & 2 & &2 && 0 && 0 &&&\\
&&&&0 &&0 && 2 &&0 && 0 &&&&\\
&&&&& 0 && 0 && 0 &&0&&& \\
&&&&&&0 &&1&&0&&&&&&\\
&&&&&&&0&&0&&&&&&&\\
&&&&&&&&1 &&&&&&&&
\end{matrix}\]}
\end{center}
\end{proposition}
From the above diamond it immediately follows that holomorphic Euler characteristics for $\TT$ are $\chi(\Omega^1_{\TT})=-1, \ \chi(\Omega^2_{\TT})=0, \ \chi(\Omega^3_{\TT})=2$. These can be easily double-checked using Macaulay2. Moreover the topological Euler characteristic $e_{\textrm{top}}(\TT)=0$ (cf. \cite{faenzi}, Ex. 4.14).
\begin{corollary} \label{cor29}Let $X= \TT(2,9) \cap H$ be a linear section of $\TT(2,9)$. This is a Fano of K3 type with Hodge diamond
\begin{center}
{\small
\[\begin{matrix}
&&&&&&&&1 &&&&&&&&\\
&&&&&&&0&&0&&&&&&&\\
&&&&&&0 &&1&&0&&&&&&\\
 &&&&& 0 && 0 && 0 &&0&&& \\
&&&&0 &&0 && 2 &&0 && 0 &&&&\\
&&&0&&0 & & 2 & &2 && 0 && 0 &&&\\
&&0 && 0 && 1 &&22 &&1 &&0 && 0&&\\
&&&0&&0 & & 2 & &2 && 0 && 0 &&&\\
&&&&0 &&0 && 2 &&0 && 0 &&&&\\
&&&&& 0 && 0 && 0 &&0&&& \\
&&&&&&0 &&1&&0&&&&&&\\
&&&&&&&0&&0&&&&&&&\\
&&&&&&&&1 &&&&&&&&
\end{matrix}\]}
\end{center}
The vanishing subspace is $h^{2,2}_{\van}(X)=20$. The holomorphic Euler characteristics for $X$ are $\chi(\Omega^1_{X})=-1, \ \chi(\Omega^2_{X})=1, \ \chi(\Omega^3_{\TT})=-18$. Moreover the topological Euler characteristic $e_{\textrm{top}}(X)=24$.\end{corollary}
\begin{proof} The Hodge numbers for $X$ follows from those of $\TT(2,9)$ together with the computations of $\chi(\Omega^i)$, which can be easily done a priori via Riemann-Roch and the help of computer algebra.
\end{proof}
\subsubsection{Borel-Bott-Weil computation for $\TT(2,9)$}
Borel-Bott-Weil theorem is a powerful tool for computing cohomologies of vector bundles on homogeneous spaces. Together with some well-known sequences it is often sufficient to compute Hodge numbers for varieties cut by general global sections of homogeneous vector bundles. Although rather long and involved, the procedure is mostly algorithmic. We will include the general setup (skipping most details for the sake of readability) in order to give the reader a toolbox for further computations.
\subsubsection*{General BBW strategy}
Let $\Gr(k,n)$ be the
Grassmannian of $k$-dimensional subspaces of $V_n$. Consider two dominant weights $\alpha
= (\alpha_1, \dots ,\alpha_{n-k})$ and $\beta = (\beta_1, \dots
,\beta_{k})$ for the Schur functors $\Sigma$ applied to $\mQ$ and $\mR$ and their concatenation $\gamma = 
(\gamma_1,\dots,\gamma_n)$. Let $\delta$ the decreasing sequence $\delta = (n-1, \dots , 0)$ and
consider $\gamma + \delta$. Write $\sort(\gamma + \delta)$ for the sequence obtained
by arranging the entries of $\gamma + \delta$ in non-increasing order, and define
$\tilde{\gamma} = \sort(\gamma + \delta)- \delta$.
If $\gamma + \delta$ has repeated entries, then
$$H^i(\Gr(k,n),\Sigma_{\alpha} \mQ \otimes \Sigma_{\beta} \mR)=0$$ for all $i
\ge 0$. Otherwise,
writing $l$ for the \emph{number of disorders}, that is the number of pairs $(i, j)$ with $1 \le i < j  \le n$
and $\gamma_i - i
< \gamma_j - j$  we have$$H^l(\Gr(k,n),\Sigma_{\alpha} \mQ \otimes
\Sigma_{\beta} \mR ) = \Sigma_{\tilde{\gamma}} V$$
and
$H^i(\Gr(k,n),\Sigma_{\alpha} \mQ\otimes \Sigma_{\beta}\mR )=0$ for $i \ne
l$. 
Let now $Z \subset \Gr(k,n)$ a variety which is the zero locus of a general section of a rank $r$ globally generated vector bundle $F^{\vee}$. We have the Koszul complex for $Z$, which is indeed a resolution
\begin{equation}\label{koszul} 0 \to \det(F) \to \W^{r-1} F \to \ldots \to F \to \of_G \to \of_Z \to 0. \end{equation}
If $H$ is another globally generated vector bundle on $\Gr(k,n)$ we can tensor the above sequence by $H$:  we have the spectral sequence
 $$\mathbf{E}_1^{-q,p} =H^p(\Gr(k,n), H \otimes \W^q F)\Rightarrow H^{p-q}(Z, H|_Z),$$
if moreover both $F$ and $H$ are homogeneous we can compute all terms on the left by BBW formula. We can now compute the Hodge numbers for our $X$.  Notice that the $F$ in the Koszul complex above is the dual of bundle we start with. In this case it will be $\mQ(-1)$.
\subsubsection*{The Hodge numbers $h^{1,i}(\TT(2,9))$}
We apply the above formula together with the conormal sequence, which since $N^{\vee}_{\TT/\Gr} \cong F$ becomes
$$ 0 \to F|_{\TT}\to \Omega^1_G|_{\TT} \to \Omega^1_{\TT} \to 0.$$
We can compute the cohomologies of the first two bundles using the above strategy. $F|_X$ turns out to be acyclic, whereas the only non-zero cohomology of $\Omega^1_G|_{\TT}$ is $H^1(\Omega^1_G|_{\TT}) \cong H^1(\Omega^1_G) \cong \C$. It follows that the Hodge numbers $h^{1,i}(\TT)=0$, $i \neq 1$ and $h^{1,1}(\TT)=1$.
\subsubsection*{The Hodge numbers $h^{2,i}(\TT(2,9))$}
In order to compute these other Hodge numbers we need to rise the conormal sequence to the second exterior power, that is 
$$ 0\to \Sym^2 F|_{\TT} \to (F \otimes \Omega^1_G)|_{\TT} \to \Omega^2_{G}|_{\TT} \to \Omega^2_{\TT} \to 0. $$
$\Sym^2 F \otimes \bigwedge^i F$ is acylic for $i\neq 7$. This can be checked using first the Littlewood-Richardson formula to determine the irreducible decomposition of each of these bundles, and then applying several iteration of the BBW formula.
For $i=7$ it is $\Sigma_{3,1^6}\mQ\otimes \Sigma_{9,9}\mR$ that has
$H^{12}(\Sym^2 F \otimes \bigwedge^7 F) \cong \mathbb{C}$ (and
therefore $H^5(\Sym^2 F|_{\TT}) \cong \mathbb{C}$). The bundle $\Omega^1 \otimes F \otimes \bigwedge^i F$ is acylic for all i.
The bundle $\Omega^2 \otimes \bigwedge^i F$ is not acylic for $i=0$ (and $H^2(\Omega^2_G|_{\TT}) \cong\C^2$) and for $i=3$.
Indeed in the case $i=3$ its decomposition in irreducibles contains the summand
$ \Sigma_{3,3,3,2,2,1,1}\mQ\otimes \Sigma_{7,5} \mR$.
This gives $H^6(\Omega^2 \otimes \bigwedge^3 F)= \mathbb{C}$.
Putting all these data together one obtains $H^2(\Omega^2_{\TT})=H^3(\Omega^2_{\TT})\cong \C^2$ with the other Hodge $h^{2,i}=0$.
\subsubsection*{The Hodge numbers $h^{3,i}(\TT(2,9))$}
By Riemann-Roch one gets $\chi(\Omega^3_{\TT})=2$. Thanks to the knowledge of $h^{i,3}(\TT)$ for $i \neq 3,4$, this implies $ h^{3,3}(\TT)=h^{4,3}(\TT).$
 We use the third power of the conormal sequence, namely
$$0 \to \Sym^3 F |_{\TT} \to (\Omega^1 \otimes \Sym^2F)|_{\TT} \to (\Omega^2 \otimes
F)|_{\TT} \to \Omega^3_G|_{\TT} \to \Omega^3_{\TT} \to 0.$$

One strategy is to split the sequence above in three short one, namely
\begin{equation}0 \to \Sym^3 F |_{\TT} \to (\Omega^1 \otimes \Sym^2F)|_{\TT} \to J_2 \to 0 ,\end{equation}
\begin{equation} 0 \to J_2 \to (\Omega^2 \otimes F)|_{\TT} \to J_1 \to 0 ,\end{equation}
\begin{equation} \label{finaleq} 0\to J_1 \to \Omega^3_G|_{\TT} \to \Omega^3_{\TT} \to 0. \end{equation}

The only cohomological contributions come from
\begin{enumerate}[(a)] 
\item $H^{12}(Sym^ 3F \otimes \bigwedge^6 F)= \mathbb{C}^{81} \cong \End(V_9) \cong \mathfrak{gl}(V_9)$;
\item $H^{12}(Sym^3 F \otimes \bigwedge^7 F) = \mathbb{C}^{84} \cong \W^3 V_9$;
\item $H^{13}(\Omega^1 \otimes Sym^2F \otimes \bigwedge^7 F) = \mathbb{C} \cong H^6 ((\Omega^1 \otimes \Sym^2F)|_{\TT})$;
\item $H^6(\Omega^2 \otimes F \otimes \bigwedge^2 F) = \mathbb{C}\cong H^4((\Omega^2 \otimes F)|_{\TT})$;
\item $H^{10}(\Omega^2 \otimes F \otimes \bigwedge^5 F)= \mathbb{C} \cong H^5((\Omega^2 \otimes F)|_{\TT})$;
\item $H^3(\Omega^3) = \mathbb{C}^2 \cong H^3(\Omega^3_G|_{\TT})$;
\item $H^7(\Omega^3 \otimes \bigwedge^3 F)= \mathbb{C} \cong H^4(\Omega^3_G|_{\TT} )$;
\item $H^{11}(\Omega^3 \otimes \bigwedge^6 F) = \mathbb{C} \cong  H^5(\Omega^3_G|_{\TT} )$.
\end{enumerate}

Except in the case of (a) and (b) one can compute immediately the cohomology of the restriction of the bundles to $\TT$. The only non obvious case is given by the exact sequence
$$ 0 \to H^5 (\Sym^3 F |_{\TT} ) \to \bigwedge^3 V \stackrel{\phi_f}{\to} \End(V_9) \to H^6 (\Sym^3 F |_{\TT}) \to 0.$$
The situation is analogous to (\cite{klm}, Appendix B). Indeed the dual of the map $\phi_f$ is the map $\varphi_f: \End(V_9) \to \W^3 V_9^{\vee}$ mapping $u \mapsto u(f)$, where $f$ is the defining section for $\TT$ and $u$ is the Lie action. This is because one can do the same computation in family, use the $GL(V)$ equivariance to ensure that $\varphi_f$ depends linearly on $f$. Since up to a scalar there is a unique equivariant map from $\W^3 V^{\vee}$ to $\Hom(\End(V), \W^3 V^{\vee})$ we can conclude. Therefore for general $f$ the map $\varphi_f$ is injective (this can be verified for example using the general form for $f$ given in (\cite{faenzi}, 4.14) with sufficiently general coefficients and therefore $\phi_f$ is surjective as required.\\
If we plug in these cohomological informations in the long exact sequence associate to the sequence \ref{finaleq} we get several non-zero cohomology groups. In particular the final groups in this sequence are
$$ \ldots \to \C \stackrel{\epsilon}{\to} H^4 (\Omega^3_{\TT}) \stackrel{\mu}{\to} \C^2 \stackrel{\nu}{\to} \C \to 0$$
Therefore $h^{3,3}(\TT)=h^{3,4}(\TT)= \ddim (\ker \mu) + \ddim (\mathrm{Im} \ \mu)$ and by standard properties of long exact sequences $h^{3,3}(\TT)=h^{3,4} \leq 2$. On the other hand by Hard Lefschetz $h^{3,3}(\TT)=h^{3,4} \geq 2$. This concludes the proof of the theorem. 
\subsubsection{Geometry of $\TT(2,9)$ and $X$}
This rather atypical (for our setting) Hodge structure for $\TT(2,9)$ has a geometrical explanation.\\
First consider a linear section $X_H \subset \Gr(3,9)$. It is a Fano 17-fold of index 8. One can compute that its central Hodge structure has level 1, with the same numerology of a genus 2 curve. Consider the configuration in the diagram below.
The map $p: \mathrm{Fl}(2,3,9) \to \Gr(3,9)$ is a $\PP^2$ bundle, given by the choice of $V_2 \subset V_3$. It remains as well a $\PP^2$ bundle if we restrict $p$ from $X_{p^* H} \to X_H$. The Hodge structure of $X_H \subset \Gr(3,9)$ is therefore repeated three times in $X_{p^* H}$. Consider as well the projection $\phi$ from $\mathrm{Fl}(2,3,9) \cong \PP_{\Gr(2,9}(\mQ(-1))$ to $\Gr(2,9)$, that is a $\PP^6$-bundle. Restricting $\phi$ to $X_{p^* H}$ this gives a $\PP^5$ bundle generically on $\Gr(2,9)$, that degenerates on a $\PP^6$ on the zero locus $Z_H$ of a section of the dual of $\mQ(-1)$, that is $\TT(2,9)$.
\begin{equation} \label{diagrammoneswaggone}\xymatrix{
& F \ar[dl]^\phi & X_{p^* H} \ar[dr]^p \ar[dl]^ \phi \ar@{^{(}->}[r] & \mathrm{Fl}(2,3,9) \ar[dr]^p \\
Z_H \ar@{^{(}->}[r] & \Gr(2,9) & &  X_H \ar@{^{(}->}[r] & \Gr(3,9)   
}\end{equation}
One can prove that the Hodge structure of $\TT(2,9)$ can be pushed down from $X_{p^* H}$, which in turn can be calculated from $X_H \subset \Gr(3,9)$. This can be considered as an alternative (and a bit more geometrical) proof of Thm. \ref{t129}. The precise details of this construction and extension to the derived category case will appear in \cite{nested}. In particular a similar argument, albeit in a more complicated version, can be used to derive directly Corollary \ref{cor29} and geometrically explain the K3 structure.
We do not produce here a result interpreting some moduli space on $X$ as an IHS: however we expect a similar result to Proposition \ref{t2} to hold here as well.

\subsection{S6: $3 \times $K3 structure }
This sporadic Fano has some interesting features. First of all, unlike all our other examples, it is not a section of another Fano by the zero locus of a line bundle. Then it is a Fano of K3 type in two different ways. \\
The variety $\TT(2,10)$ is the zero locus of a general global section of the bundle $\mQ^*(1)$ on the Grassmannian $\Gr(2,10)$. As in the previous case \textbf{S5} we have  $$H^0(\Gr(2,10), \mQ^*(1)) \cong \W^3 V_{10}^{\vee},$$ therefore $\TT(2,10)$ is given by the locus of two-spaces in a 10-dimensonal space which are annihilated by a 3-form. It is straightforward to check that $\TT(2,10)$ is a Fano 8-fold of index $\iota=3$. We compute first its Hodge numbers

\begin{proposition} \label{3k3}The Hodge numbers of $\TT(2,10)$ are
\begin{center}
{\small
\[\begin{matrix}
&&&&&&&&1 &&&&&&&&\\
&&&&&&&0&&0&&&&&&&\\
&&&&&&0 &&1&&0&&&&&&\\
 &&&&& 0 && 0 && 0 &&0&&& \\
&&&&0 &&0 && 2 &&0 && 0 &&&&\\
&&&0&&0 & & 0 & &0 && 0 && 0 &&&\\
&&0 && 0 && 1 &&22 &&1 &&0 && 0&&\\
&0 && 0 && 0 && 0&& 0 &&0 &&0 && 0&\\
0 && 0 && 0 && 1&& 23 &&1 &&0 && 0&&0\\
&0 && 0 && 0 && 0&& 0 &&0 &&0 && 0&\\
&&0 && 0 && 1 &&22 &&1 &&0 && 0&&\\
&&&0&&0 & & 0 & &0 && 0 && 0 &&&\\
&&&&0 &&0 && 2 &&0 && 0 &&&&\\
&&&&& 0 && 0 && 0 &&0&&& \\
&&&&&&0 &&1&&0&&&&&&\\
&&&&&&&0&&0&&&&&&&\\
&&&&&&&&1 &&&&&&&&
\end{matrix}\]}
 \end{center}
\end{proposition}

As we can see from the above theorem, $\TT(2,10)$ has a Hodge structure of K3 type both in $H^6$ (and therefore in $H^{10}$ by duality) and in $H^8$, making it a rather peculiar example. Indeed by Hard Lefschetz the K3 structure in $H^6$ immediately implies the presence of a K3 sub-structure in $H^{10}$. The surprising bit is that this is the whole of $H^8$, with the exception of a primitive cycle. The computation of the above Hodge numbers is done via a Borel-Bott-Weil computation, as in the previous section. Since these are rather long computations (and not really different from the previous case) we will just sketch it. 
\begin{proof} Let $F$ be the dual of the bundle that cuts $\TT$.The computations of the Hodge numbers until $h^{2,i}$ does not present any challenge. In the third exterior power of the conormal exact sequence 
$$0 \to \Sym^3 F |_{\TT} \to (\Omega^1 \otimes \Sym^2F)|_{\TT} \to (\Omega^2 \otimes
F)|_{\TT} \to \Omega^3_G|_{\TT} \to \Omega^3_{\TT} \to 0$$
we have that $\Omega^2 \otimes
F)|_{\TT}$ is acylic, for $(\Omega^1 \otimes \Sym^2F)|_{\TT}$ the unique cohomology group is $H^7((\Omega^1 \otimes \Sym^2F)|_{\TT}) \cong \C$ and for the third cotangent we have $H^3( \Omega^3_G|_{\TT}) \cong \C^2$. The only tricky part comes when considering $\Sym^3 F |_{\TT}$. Indeed from the spectral sequence associated to the Koszul resolution for $\Sym^3 F |_{\TT}$ one finds an exact sequence 
$$ 0 \to H^{13}(K_7) \to H^{14}( \W^8 F \otimes \Sym^3 F) \to   H^{14}( \W^7 F \otimes \Sym^3 F) \to H^{14}(K_7) \to 0 $$ 

 where $K_7$ is the sheaf which we used to complete the sequence $0 \to \W^8 F \otimes \Sym^3 F \to  \W^7 F \otimes \Sym^3 F$. The above sequence is equal to:
 $$ 0 \to H^{13}(K_7) \to \W^3 V_{10} \to   \End(V_{10}) \to H^{14}(K_7) \to 0 $$ 
 As in the previous section case, one can argue that the middle map is surjective, and therefore chasing the sequence one gets that the unique cohomology group for  $\Sym^3 F |_{\TT}$ is $H^6( \Sym^3 F |_{\TT}) \cong \C^{20}$. Collecting all these data together in the above long exact sequence we get $h^{3,3}(\TT)=22$ and $h^{5,3}(\TT)=1$. The missing number can be obtained from the computation of the Euler characteristic.
\end{proof}

This strange Hodge structure can be explained with a construction absolutely equivalent to the one of \eqref{diagrammoneswaggone}, with of course  $\mathrm{Fl}(2,3,10)$ being the relevant Flag.  In particular, one can repeat the construction of 3.9.2 and do the computations in $K_0(\textrm{Var})$ as an alternative way of computing Hodge numbers. Indeed this is the same Hodge structure coming from the Debarre-Voisin twentyfold $Y_1 \subset \Gr(3,10)$. It is therefore not surprising that we can relate the IHS fourfold $Z_{DV} \subset \Gr(6,10)$ to $\TT(2,10)$. \\
Define first $Z_{\of(1)^4}$ to be the zero locus of four general linear sections in the Grassmannian $\Gr(2,6)$. Moreover we denote by  $\TT_{,\omega}(2,10)$ a distinguished element of the family defined by a specified 3-form $\omega$.
\begin{proposition} \label{t2}
The Debarre-Voisin fourfold $F_{\omega}$ is birational to the moduli space (contained in the Hilbert scheme) of fourfolds $Z_{\of(1)^4}$ contained in the variety $\TT_{,\omega}(2,10)$.
\end{proposition} 
\begin{proof}
Let $W$ be a general point in the Debarre-Voisin fourfold given by a general three form $\omega$. Let us consider the subscheme of $\TT_{,\omega}(2,10)$ given by all two spaces contained inside $W$. This does not coincide with the full Grassmannian $\Gr(2,6)$, as the condition $\omega(W)=0$ does not imply $\omega \contr \W^2 U=0$ for all $U\subset W$ two-spaces. Notice that this is not the case if one considers three spaces contained in $W$, that is the construction of the Debarre-Voisin IHS fourfold as a moduli space of $\Gr(3,6)$ contained in the respective twentyfold.\\
On $\Gr(k,10)$ for all $k$ we have a sequence $0\to \mR \to V_{10}\otimes\mathcal{O}\to (V_{10}\otimes\mathcal{O})/\mR\to 0$ which dually gives a sequence $0 \to \mR^\perp \to V_{10}^\vee\otimes\mathcal{O}\to \mR^\vee \to 0$. This gives a filtration of $\W^3 V_{10}^\vee\otimes\mathcal{O}$ with factors $\W^3 \mR^\perp,$ $\W^2\mR^\perp\otimes \mR^\vee$, $\mR^\perp\otimes \W^2\mR^\vee$ and $\W^3\mR^\vee$. The three-form $\omega$ is a section of the last factor $\W^3\mR^\vee$ on $\Gr(6,10)$. On the zero locus of such a section, this lift to a section of $\mR^\perp\otimes\W^2\mR^\vee$, which corresponds to a map 
$V_{10}/W\rightarrow \W^2 W^\vee.$
The image of such a map is a four dimensional space $H_4$ of two forms on $W$, for every six space $W$ in the Debarre-Voisin twentyfold given by $\omega$.\\
Let $U\subset W$ be a point of $\TT_{,\omega}(2,10)$. The space $U$ is isotropic for all two forms in $H_4$, indeed if this were not the case we would have a two form $\sigma\in H_4$ such that $\sigma_{|U}$ is non degenerate and, by how forms in $H_4$ are obtained, this would imply $\omega \contr \W^2 U\neq 0$. On the contrary, in an appropriate basis, it is not difficult to show that $\omega \contr \W^2 U =0$ is implied by $\sigma(U)=0$ for all $\sigma\in H_4$. Thus, the scheme of subspaces $U\subset W$ with fixed $W$ is parametrized by a fourfold $Z_{\of(1)^4} \subset \Gr(2,W)$, which a Fano fourfold of index two, rational by \cite[Thm. 2.2.1]{fei},  with central cohomology $(h^{1,1}, h^{2,2})=(1,8)$. This gives a rational map between the Debarre-Voisin fourfold and the space of $Z_{\of(1)^4}$ contained in $\TT(2,10)$ (and in a fixed $\Gr(2,6)$). As by changing the point of the Debarre Voisin fourfold we change the ambient Grassmannian $\Gr(2,6)$, it is clear that such a map is generically injective, hence birational.
\end{proof}

\subsection{S7: a mixed (2,3) CY structure}
A curious yet interesting thing happens when we take a linear section $X_H$ of the above $\TT(2,10)$. Indeed by Lefschetz's hyperplane section theorem we know that the K3 structure of $\TT(2,10)$ in $H^6$ and $H^8$ must transfer to its linear section: what is most interesting is that the $H^7$ presents as well a Calabi-Yau structure of level three. To the best of our knowledge, this is the first example of a prime variety that has 2 different examples of CY-structure, of course in different weights. The precise result is 
\begin{proposition}\label{23cy} The Hodge numbers of a linear section $X_H \subset \TT(2,10)$ are
\begin{center}
{\small
\[\begin{matrix}
&&&&&&&&1 &&&&&&&&\\
&&&&&&&0&&0&&&&&&&\\
&&&&&&0 &&1&&0&&&&&&\\
 &&&&& 0 && 0 && 0 &&0&&& \\
&&&&0 &&0 && 2 &&0 && 0 &&&&\\
&&&0&&0 & & 0 & &0 && 0 && 0 &&&\\
&&0 && 0 && 1 &&22 &&1 &&0 && 0&&\\
&0 && 0 && 1 && 44&& 44 &&1&&0 && 0&\\
&&0 && 0 && 1 &&22 &&1 &&0 && 0&&\\
&&&0&&0 & & 0 & &0 && 0 && 0 &&&\\
&&&&0 &&0 && 2 &&0 && 0 &&&&\\
&&&&& 0 && 0 && 0 &&0&&& \\
&&&&&&0 &&1&&0&&&&&&\\
&&&&&&&0&&0&&&&&&&\\
&&&&&&&&1 &&&&&&&&
\end{matrix}\]}
\end{center}
\end{proposition}
The above proposition can be proved with a Borel-Bott-Weil computation similar to the ones above. We will not add further details here in order to preserve the readability of the current paper. We will indeed give a sketch of a geometrical explanation of why such numbers appear.\\
Indeed as an expert reader might notice, the 3CY structure in our $X_H$ has the same dimension of the 3CY structure appearing in the $H^{23}$ of a linear section $X_1 \subset \Gr(3,11)$.  We will give now an explanation on why and how this 3CY structure gets projected from such varieties to our $X_H \subset \TT(2,10)$. This will be only sketched, since the details (in a more general context) will appear in the forthcoming \cite{nested}. The first steps are the following lemmata.
\begin{lemma} A linear section $X_1 \subset \Gr(3,11)$ is a Fano 23-fold of 3CY type. Indeed its non-zero Hodge numbers of weight 23 are $(h^{10,13},h^{11,12}, h^{12,11}, h^{13,10})=(1,44,44,1)$.
\end{lemma}
This lemma can easily be proved, for example using our results in \cite{eg1}. We notice that such a variety is of 3CY even in the (stronger) categorical sense, see \cite[4.5]{kuzicy}. The orthogonal complement to the Calabi-Yau category is generated by 150 exceptional objects. The following Lemma is less obvious
\begin{lemma}\label{symphodge} A linear section $Y_1 \subset \SGr(3,10)$ is a Fano 17-fold of 3CY type. Indeed its non-zero Hodge numbers of weight 17 are $(h^{7,10},h^{8,9}, h^{9,8}, h^{10,7})=(1,44,44,1)$.
\end{lemma}
It can be proved for example with similar calculations to Corollary \ref{hodges2}, since we already know that the symplectic Grassmannian $\SGr(k, n)$ is a central variety. However one can prove that the statement is more than merely numerological. Indeed one can show the existence of a fully faithful functor $\Phi: D^b(Y_1) \to D^b(X_1)$ and a semiorthogonal decomposition of $D^b(X_1)$ with $\Phi D^b(Y_1)$ as first component, together with a bunch of exceptional objects. This obviously proves the Hodge-theoretical statement as well. This in turn explains the 3CY structure in $X_H \subset \TT(2,10)$. Indeed it is possible to write a diagram like the one for $\TT(2,9)$ in \ref{diagrammoneswaggone}, appropriately modified; in particular we have to pass through the symplectic partial flag $\mathrm{SFL}(2,3,9)$. The construction is more involved, but it is enough to explain that this mixed (2,3) Calabi-Yau structure ultimately comes from an hyperplane section of (respectively) $\Gr(3,10)$ and $\Gr(3,11)$.  An interesting problem is therefore to look for other examples of varieties with mixed CY structure that are not induced by these constructions tricks outlined in \cite{nested}.  

\subsection{S8: other K3 structures as $X_L \subset \TT(k,10)$} 
A similar construction can be applied to $\TT(4,10)$, $\TT(5,10)$ and their linear sections. Indeed both of them will inherite a bunch of K3 type structure as in \ref{diagrammoneswaggone}. As an example, in the case of $\TT(4,10)$ the diagram will be 
\begin{equation}\xymatrix{
& F \ar[dl]^\phi & X_{\pi^* H} \ar[dr]^p \ar[dl]^ \phi \ar@{^{(}->}[r] & \mathrm{Fl}(3,4,10) \ar[dr]^p \\
Z_H \ar@{^{(}->}[r] & \Gr(4,10) & &  X_H \ar@{^{(}->}[r] & \Gr(3,10)   
}\end{equation}
The map $p$ is a $\PP^6$ bundle, whereas $\phi$ is generically a $\PP^2$ bundle specialising to a $\PP^3$ bundle over $Z_H$. This suggests that $\TT(4,10)$ should have  7xK3 type structure, and a Borel-Bott-Weil calculation confirms this. A similar construction, albeit more complicated can be performed as well for $\TT(5,10)$, where the fibers of the map on the right hand side of the diagram are $\Gr(2,7)$. Moreover on the left side of the diagram there are three type of fibers, corresponding to (generically) smooth hyperplane sections of $\Gr(2,5)$, singular sections in codimension 3 and the whole of $\Gr(2,5)$ in codimension 10. Of course the linear sections of both $\TT(4,10)$ and $\TT(5,10)$ inherits some structure of K3 type by Lefschetz theorem (depending of course by the dimension of the linear subspace). It is interesting to notice codimensional 1 and 2 linear section will be of mixed CY type, with an argument equivalent to the one of the previous section. Finally we remark that even $\TT(6,10)$ and $\TT(1,10)$ admits structures of K3 type: the first one is nothing but the IHS fourfold of Debarre-Voisin, while a the second one can be used to construct the Peskine variety in $\PP^9$, although formally the latter is given by a degeneracy locus. In \cite{nested} we will indeed use this approach to compute the Hodge numbers of this special variety. 
\section*{Appendix A: some extra Fano of 3-CY type}
The methods of this paper can be used to produce Fano of $k$- CY type for every $k$. Another interesting case is when the variety is 3CY. This has been already considered by Iliev and Manivel in \cite{ilievmanivel}. They classified the Fano of 3CY type that can be obtained as linear or quadratic section of homogeneous space, under the additional assumption that the $H^1(T_X)$ was to be isomorphic to one of the Hodge groups of $X$. Many more examples can be found using our method, especially if  this condition is not assumed. We do not write the full list here, since we believe it would not fit well with rest of the paper. However it is worthy to point out that many of the examples can be produced as linear sections of symplectic and bisymplectic Grassmannian, with an explanation as in Lemma \ref{symphodge}. \\
Indeed such examples include $ X_1 \subset \SGr(3,10)$ and $X_1 \subset \SGr(4,9)$ in the symplectic Grassmannian and $X_1 \subset \Ml(3,9), X_1 \subset \Ml(4,9)$ and $X_2 \subset \Ml(2,6)$ for the bisymplectic. We point out that the Hodge structure of the linear section of $\SGr(3,10)$ and $\Ml(3,9)$ comes from an hyperplane section of $\Gr(3,11)$ (which is as well of 3CY type) with an argument similar to Lemma \ref{symphodge} to be fully spelled out in \cite{nested}. A different but not dissimilar argument can be made for $X_2 \subset \Ml(2,6)$ and explain how this structure of 3CY comes from $X_2 \subset \Gr(2,6)$. In the symplectic Grassmannian we find as well $X_2 \subset \SGr(4,7)$. In the Orthogonal Grassmannian we find the examples of linear sections of $\OGr(3,9), \OGr(4,9)$ and $\OGr^+(5,10)$. The latter is equivalent to a quadratic section of $\mathbb{S}_{10}$ in the spinor embedding (since the line bundle required is the square root of the Pl\"ucker one). This is already in the list of Iliev and Manivel, so we will not include it.\\
Another interesting example is a section $X_H \subset \mathrm{SO}(3,8)$ the \emph{ortho-symplectic} Grassmannian. The latter is given by the zero locus of $\W^2 \mR^\vee \oplus \Sym^2 \mR^\vee$ on $\Gr(3,8)$. We use the notation $X_H$ to point out that, as in the case of Orthogonal Grassmannian $\OGr(3,8)$, $\mathrm{SO}(3,8)$ has Picard rank equal to 2. We checked that there are no other examples of Fano of CY3 type in the orthosymplectic Grassmannian. \\
The cohomology of the orthosymplectic Grassmannian can be computed using a torus action on it (as remarked also in \cite{benphd}, and then Lefschetz's theorem and Borel-Bott-Weil theorem allow us to compute the cohomology of its linear sections in many cases. First, notice that two general symmetric and skew symmetric forms $s,\lambda$ on a space of dimension $2n$ can be put in the following canonical form:
$$s =\sum^{2n} s_ix_i^2; \ \  \lambda =\sum^{n} x_{2i}\wedge x_{2i+1}.$$
In this way, the stabilizer of these forms contains as maximal torus the group $(\C^*)^n$. In a similar fashion to \cite[Prop 4.2.1]{benphd}, one can prove that this maximal torus has only isolated fixed points (more precisely, $2^k{n\choose k}$) and therefore the cohomology of the orthosymplectic grassmannian is concentrated in the $(p,p)$ part (and the characteristic of the cotangent sheaf and its exterior powers give us the desired cohomology).\\
From this, we obtain the following cohomology for $X_H$:
\begin{center}
{\small
\[\begin{matrix}
&&&&&&&&1 &&&&&&&&\\
&&&&&&&0&&0&&&&&&&\\
&&&&&&0 &&2&&0&&&&&&\\
 &&&&& 0 && 0 && 0 &&0&&& \\
&&&&0 &&0 && 7 &&0 && 0 &&&&\\
&&&0&&1 & & 45 & &45 && 1 && 0 &&&\\
&&&&0 &&0 && 7 &&0 && 0 &&&&\\
&&&&& 0 && 0 && 0 &&0&&& \\
&&&&&&0 &&2&&0&&&&&&\\
&&&&&&&0&&0&&&&&&&\\
&&&&&&&&1 &&&&&&&&
\end{matrix}\]}
\end{center}
We point out that $X_H \subset \mathrm{SO}(3,8), \ X_2  \subset \SGr(4,7)$ and $X_2 \subset \Ml(2,6)$ are particularly interesting as Fano of 3CY type, since they are of dimension 5 (the minimal possible) and therefore relevant for testing a modified version of Kuznetsov's conjecture on rationality and derived categories. We collect in the next table the 3CY structure mentioned in the above discussion. We mention that in analogy with the K3 case, $\TT(2,11)$ can be considered as well as an example of 3$\times$CY structure. Taking other $k$ and appropriate number of linear sections is possible as well to produce examples of \emph{mixed} $(3,j)$ CY structure. However we do not include them in the following table. Of course more examples could be found by considering products and such as in the FK3 case, but we decided to not consider them here in order to keep the length of this paper within an acceptable limit. We of course do not include the examples already considered in \cite{ilievmanivel}.

\begin{center}
\begin{tabular}{@{} *9l @{}l @{}l @{}l @{}l @{}l @{}}    \toprule
Type &dim. &$\iota_X$& $h^{n-1/2,n+1/2}$ \\ \midrule
$X_1 \subset \OGr(3,9)$ & 11 &4& 49\\
$X_1 \subset \OGr(4,9)$ & 9 &3&70\\
$X_1 \subset \SGr(3,10)$ & 17 &7&44 \\
$X_1 \subset \SGr(4,9)$ & 13 &5&45 \\
$X_1 \subset \Ml(3,9)$ & 11 &4& 44\\
$X_1 \subset \Ml(4,9)$ & 7 &2&45 \\
$X_2 \subset \Ml(2,6)$ & 5 &2&67 \\
$X_H \subset \mathrm{SO}(3,8)$ & 5 &1& 45 \\
$X_2 \subset \SGr(4,7)$ & 5 &2& 72 \\


    \bottomrule
 \hline
\end{tabular}
\captionof{table}{3CY structure in $\OGr, \SGr, \Ml$ and $\mathrm{SO}$} \label{table3}
\end{center} 

\section*{Appendix B: infinite CY series}
During our search we identified some interesting class of varieties. Although not directly related to the main story of this paper, they have some interesting features that is worth to underline. As an example we identified some interesting infinite families of varieties (of every dimension) with trivial canonical bundle obtained using the same bundles in different Grassmannians. We checked that these varieties are actually Calabi-Yau for low dimension (up to 6). We expect them to be always like this. \\
We describe now these series of varieties, according to the type of bundles involved.
\begin{align*} &A(k,l):= \mQ(1) \oplus \W^2\mR^{\vee}  \textrm{ on } \Gr(k, k+l);\\
&B(k,l):=  \mQ^{\vee}(1) \oplus \Sym^2 \mR^{\vee}\textrm{ on } \Gr(k, k+l);\\
&C(k, k+1):=\Sym^2 \mR^{\vee}\oplus \W^2\mR^{\vee}\oplus\mathcal{O}(1) \textrm{ on } \Gr(k,2k+1).
\end{align*}
$A(k,l)$ has dimension $l(k-1)-{k \choose 2}$, $B(k,l)$ has dimension $l(k-1)-{k+1 \choose 2}$ and $C(k)$ has dimension $k-1$. Notice that $A(k,l)$ can naturally be seen as $Z_{\mQ(1)} \subset \SGr(k,k+l)$, $B(k,l)$ as $Z_{\mQ^{\vee}(1)} \subset \OGr(k, k+l$) and $C$ is a linear section of the ortho-symplectic Grassmannian. In particular, as in \cite{kuznetsovc5} in the case of bisymplectic Grassmannian, one can prove that \[ C(k,k+1) \cong X_{(2, \ldots, 2)} \subset (\PP^1)^k.\]\\
When $k=2$, the zero locus of a general global section of $A(2,l)$ is indeed deformation of  a complete intersection given by $(\of(1))^{l+3} $ on $\Gr(2, l+3)$. Indeed first notice that  on $\Gr(2, l+3)$ we have $(\of(1))^{l+3} \cong \mQ(1) \oplus \mR(1)\cong \mQ(1) \oplus \mR^{\vee} $. Then notice that the zero locus of a general global section of $ \mQ(1) \oplus \mR^{\vee}$ on $\Gr(2, l+3)$ is isomorphic to the zero locus of a general global section of $\mQ(1) \oplus \of(1)$ on $\Gr(2, l+2)$. This follows from the standard fact that $Z_{\mR^{\vee}} \subset \Gr(k, n) \cong \Gr(k, n-1)$ and under the following isomorphism $\mQ(1)_{k,n}$ projects to $\mQ(1)_{k,n-1} \oplus \of(1)$ .\\
For dimension $d=2,3,4$ we refer to \cite{benedetti}, \cite{inoue2016complete}. For $d=5,6$ the Calabi-Yaus in the series $A$ and $B$ are the following. We do not include $B(5,5)$, since it can bee seen as a deformation of the double spinor variety studied by Manivel in \cite{manivelspinor}. The Hodge numbers are computed in Proposition 3.3, together with the fact that this family is locally complete. Since $C$ is indeed a well-known class of variety in disguise, we will not compute the Hodge numbers for the first values of the series. In the following list of invariants we do not include either trivially known Hodge numbers such as $h^{0, n}$. Moreover the number not listed are always $0$.

\begin{center}
\begin{tabular}{@{} *9l @{}l @{}l @{}l @{}l @{}l @{}}    \toprule
dim. & Type &$h^{1,1}$& $h^{2,2}$&$h^{4,1}$&$h^{3,2}$\\ \midrule
5 & $A(2,6)$ & 1&2& 163 & 1784 \\
5 & $A(3,4)$ & 1&2& 148&1619 \\
5 & $B(4,5)$ & 1&2& 165&1806 \\

    \bottomrule
 \hline
\end{tabular}
\captionof{table}{First values of infinite series for fivefolds} \label{table5}
\end{center} 

\begin{center}
\begin{tabular}{@{} *9l @{}l @{}l @{}l @{}l @{}l @{}}    \toprule
dim. & Type &$h^{1,1}$& $h^{2,2}$&$h^{5,1}$&$h^{4,2}$&$h^{3,3}$\\ \midrule

6 & $A(2,7)$ & 1&2& 251&5202&14004 \\
6 & $A(4,4)$ & 1&1& 251&5181&13960\\
6 & $B(2,9)$ & 1&2& 120 &2254&6274\\
6 & $B(3,6)$ & 1&2& 125&2380&6596 \\

    \bottomrule
 \hline
\end{tabular}
\captionof{table}{First values of infinite series for sixfolds} \label{table4}
\end{center} 

An interesting question, which however falls beyond the scope of this paper, is to investigate whether the varieties constructed in this way are generic in moduli, that is whether all of their deformations are embedded in the same Grassmannian. This can be done by a direct computation of $h^1(TG|_X)$ using Koszul complex and Borel-Bott-Weil theorem, however these calculations are quite demanding in each specific case, and a general argument is out of reach.

 \section*{Appendix C: fake K3 structure}
The numerical condition in \eqref{condition} restricted most of our search to vector bundles in which one of the irreducible summand is linear. One can of course try to rearrange this condition in order to eliminate the constraint. Indeed this is geometrically meaningful, as for example $\TT(2,10)$ shows (it is a zero locus of an indecomposable bundle that is non-linear, with slope $\mu=c_1(E)/r(E)=7/8$). It is possible, and we plan to do so, to fully investigate this case. \\
During a preliminary search we found this example, the zero locus $X_{\mR^\vee(1)} \subset \Gr(2,6)$. It is a sixfold of index 3, defined by a bundle of slope $\mu=3/2$, satisfying all our preliminary numerological condition. Although it is not of K3 type, it is rather curious, and we decided to add it anyway. Indeed its Hodge numbers are

\begin{center}
{\small
\[\begin{matrix}
&&&&&&&&1 &&&&&&&&\\
&&&&&&&0&&0&&&&&&&\\
&&&&&&0 &&1&&0&&&&&&\\
 &&&&& 0 && 0 && 0 &&0&&& \\
&&&&0 &&0 && 2 &&0 && 0 &&&&\\
&&&0&&0 & & 0 & &0 && 0 && 0 &&&\\
&&0 && 0 && 0 &&22 &&0 &&0 && 0&&\\
&&&0&&0 & & 0 & &0 && 0 && 0 &&&\\
&&&&0 &&0 && 2 &&0 && 0 &&&&\\
&&&&& 0 && 0 && 0 &&0&&& \\
&&&&&&0 &&1&&0&&&&&&\\
&&&&&&&0&&0&&&&&&&\\
&&&&&&&&1 &&&&&&&&
\end{matrix}\]}
\end{center}

The absence of the 1 in $h^{4,2}$ is explained by a Borel-Bott-Weil computation, since an inconvenient cancellation in the spectral sequence occurs. It is possible that some higher-dimensional analogue of this \textit{false positive} may occur, although we expect this to be quite an exception and not the general rule.

\end{document}